\documentclass[reqno, 10pt, a4paper]{amsart}

\usepackage{fullpage}
\selectfont

\usepackage[utf8]{inputenc}
\usepackage[OT2,T1]{fontenc}
\usepackage[english]{babel}

\usepackage{amsmath}
\usepackage{amsfonts}
\usepackage{amssymb}
\usepackage{amsthm}

\usepackage{enumitem}
\usepackage{multirow}

\usepackage{breqn}
\newcounter{myequation}[equation]

\usepackage{hyperref}

\theoremstyle{plain}
\newtheorem{theorem}{Theorem}[section]
\newtheorem{conjecture}[theorem]{Conjecture}

\newtheorem{proposition}[theorem]{Proposition}
\newtheorem{lemma}[theorem]{Lemma}
\newtheorem{corollary}[theorem]{Corollary}

\theoremstyle{definition}
\newtheorem{definition}[theorem]{Definition}

\theoremstyle{remark}
\newtheorem{remark}[theorem]{Remark}
\newtheorem{example}[theorem]{Example}
\newtheorem{fact}[theorem]{Fact}

\numberwithin{equation}{section}

\def\epsilon{\varepsilon}

\def\Magma{\textsc{Magma}}

\def\ie{i.e.}

\def\PHM{toggle model}

\DeclareMathOperator{\Aut}{Aut}

\DeclareMathOperator{\Char}{char}

\DeclareMathOperator{\DOv}{\underline{DO}}
\DeclareMathOperator{\DO}{DO}

\DeclareMathOperator{\Frac}{Frac}

\DeclareMathOperator{\GL}{GL}

\DeclareMathOperator{\PGL}{PGL}
\DeclareMathOperator{\PSL}{PSL}

\DeclareMathOperator{\Proj}{Proj}

\DeclareMathOperator{\SL}{SL}
\DeclareMathOperator{\SO}{SO}

\DeclareMathOperator{\Spec}{Spec}

\DeclareMathOperator{\Sym}{Sym}

\newcommand{\dq}{{/\kern -3pt/}}
\newcommand{\ds}{{\mathrm{ss}}}

\def\m{\mathfrak{m}}

\def\p{\mathfrak{p}}

\def\C{\mathbb{C}}

\def\F{\mathbb{F}}

\def\O{\mathcal{O}}
\def\P{\mathbb{P}\,}
\def\Q{\mathbb{Q}}

\def\Z{\mathbb{Z}}

\usepackage{bm}

\def\Mb{\bm{M}}

\def\Xb{\bm{X}}
\def\Yb{\bm{Y}}
\def\Zb{\bm{Z}}

\def\xs{\mathsf{x}}

\def\ys{\mathsf{y}}

\def\Oc{\mathcal{O}}
\def\Cc{\mathcal{C}}

\def\xfrak{{\scriptstyle \mathfrak{X}}}

\def\dv{\underline{d}}
\def\ivv{\underline{\iota}}
\def\iv{\iota}
\def\Iv{\underline{I}}

\def\xv{\underline{x}}

\hypersetup{
  pdfauthor   = {Lercier, Liu, Lorenzo Garc\'ia, Ritzenthaler},
  pdftitle    = {Reduction type of smooth plane quartics},
  pdfsubject  = {},
  pdfkeywords = {},
  backref=true, pagebackref=true, hyperindex=true, colorlinks=true,
  breaklinks=true, urlcolor=blue, linkcolor=blue, citecolor=blue,
  bookmarks=true, bookmarksopen=true}

\begin{document}

\title{Reduction type of smooth plane quartics}
\date{\today}

\begin{abstract}
  Let $C/K$ be a smooth plane quartic over a discrete valuation field. We
  characterize the type of reduction (\ie~smooth plane quartic, hyperelliptic
  genus 3 curve or bad) over $K$ in terms of the existence of a special plane
  quartic model and, over $\bar{K}$, in terms of the valuations of certain
  algebraic invariants of $C$ when the characteristic of the residue field is
  not $2,\,3,\,5$ or $7$.  On the way, we gather several results of general
  interest on geometric invariant theory over an arbitrary ring $R$ in the
  spirit of \cite{Se}. For instance when $R$ is a discrete valuation ring, we
  show the existence of a homogeneous system of parameters over $R$. We
  exhibit explicit ones for ternary quartic forms under the action of
  $\SL_{3,R}$ depending only on the characteristic $p$ of the residue
  field. We illustrate our results with the case of Picard curves for which we
  give simple criteria for the type of reduction.
\end{abstract}

\author[Lercier]{Reynald Lercier}
\address{%
  Reynald Lercier,
  DGA \& Univ Rennes, %
  CNRS, IRMAR - UMR 6625, F-35000
  Rennes, %
  France. %
}
\email{reynald.lercier@m4x.org}

\author[Liu]{Qing Liu}
\address{%
  Qing Liu,
  Institut de mathématiques de Bordeaux, CNRS - UMR 5251,
  Université de Bordeaux,
  F-33405 Talence, %
  France. %
}
\email{qing.liu@u-bordeaux.fr}

\author[Lorenzo]{Elisa Lorenzo Garc\'ia}
\address{%
  Elisa Lorenzo Garc\'ia,
  Univ Rennes, CNRS, IRMAR - UMR 6625, F-35000
  Rennes, %
  France. %
}
\email{elisa.lorenzogarcia@univ-rennes1.fr}

\author[Ritzenthaler]{Christophe Ritzenthaler}
\address{%
  Christophe Ritzenthaler,
  Univ Rennes, CNRS, IRMAR - UMR 6625, F-35000
  Rennes, %
  France. %
}
\email{christophe.ritzenthaler@univ-rennes1.fr}

\thanks{The authors would like to thank Jeroen Sijsling for his careful reading of an earlier version of the paper, Matthieu Romagny for useful references, Bianca Viray and Armand Brumer for inspiring discussions, and the referee for his/her comments improving the presentation of the manuscript. The second named author is partly supported by Xiamen University and the NSFC grant no. 11031004.
}


\maketitle

\section{Introduction and main results}

Let $K$ be a discrete valuation field with valuation $v$, valuation ring $\O$
and a uniformizer $\pi$. Let $k=\Oc/\langle \pi \rangle$ be the residue field
of characteristic $p \geq 0$.  When $F$ is an integral polynomial, \ie~with
coefficients in $\O$, we denote by $\bar{F}$ its reduction modulo $\pi$.  A
smooth projective geometrically connected curve of genus $3$ is either a plane
quartic or a hyperelliptic curve depending on whether the canonical divisor is
very ample or not.  Let $C/K$ be a smooth plane quartic.
\begin{itemize}
\item We say that $C$ \emph{has good (resp. good quartic, resp. good
    hyperelliptic) reduction} over $K$ if $C$ has a projective smooth model
  $\Cc$ over $\O$ (resp. whose special fiber is a smooth plane quartic,
  resp. a hyperelliptic curve).
\item We say that $C$ \emph{has potentially good (resp. potentially good
    quartic, resp. potentially good hyperelliptic) reduction} if for some
  finite separable extension $(K', v')$ of $(K, v)$, $C_{K'}/K'$ has good
  (resp. good quartic, resp. good hyperelliptic) reduction over $K'$.
\item We say that $C$ \emph{has bad reduction} over $K$ if it does not have
  good reduction over $K$.
\end{itemize}
Note that the type of the potential reduction does not depend on the choice of
$(K', v')$ and is given by the type of reduction of the stable model of $C$
over some finite separable extension of $K$. Hence, in the following, the
expression ``after a possible finite extension of $K$'', or
\emph{potentially}, means that we are allowed to take a finite extension of
$K$ and still call $K, \O, v, \pi$ the corresponding notions.  Let us point
out already here that when using the adjective (semi)-stable alone, we always
mean in the Deligne-Mumford sense of (semi)-stability. Otherwise, while
talking later about Geometric Invariant Theory (GIT) stability, we will always
write \emph{GIT-(semi-)stable}.

Given such a curve $C/K$, we want to determine the reduction type of $C$ among
the above possibilities.  In the first two cases, we also look for an explicit
equation of the special fiber. In the first case, it will be again a smooth
plane quartic over $k$, whereas in the second case it will be isomorphic over
$\bar{k}$ to $y^2=f(x,z)$ where $f$ is a binary octic with no multiple roots.

\begin{example} \label{ex:x13}
Let us consider the plane smooth quartic $C/\Q : F=0$ where
$$F=(x_2 +x_3) x_1^3-(2 x_2^2 + x_3 x_2) x_1^2+(x_2^3-x_3 x_2^2+2 x_3^2 x_2-x_3^3) x_1-2 x_3^2 x_2^2+3 x_3^3 x_2.$$
This curve is an equation of $X_{ns}(13) \simeq X_s(13) \simeq X_0^{+}(169)$
which is studied in \cite[Cor.6.8]{BDMTV17}. By using a general result
from~\cite{edixhoven90}, they prove that this curve has good quartic reduction
away from $13$ and potentially good hyperelliptic reduction at 13 after a
ramified extension of degree 84. By using the characterizations we are giving
in this article and our implementations, we can automatically check that the
curve has potentially good hyperelliptic reduction at $13$ and that the
special fiber is the hyperelliptic curve of affine equation $y^2=x^7-1$ (see
Example~\ref{ex:bass}).
\end{example}

The most complicated task will be to distinguish good hyperelliptic reduction
from bad reduction.  By identifying a plane quartic with its $15$ coefficients
up to a multiplicative constant, we can view the set $\Xb$ of smooth plane
quartics as a subset of $\P^{14}$.  Since isomorphisms between smooth plane
quartics are linear, $\Xb/\SL_3(\bar{K})$ gives a description of the locus
$\Mb_3^{\textrm{qrt}}$ of plane quartics inside the moduli space $\Mb_3$ of
all curves of genus $3$. Considering the locus of hyperelliptic curves as ``a
boundary'' of the locus of $\Mb_3^{\textrm{qrt}}$ inside $\Mb_3$, the
challenge is to understand whether a family of plane smooth quartics has limit
in this boundary. Similar situations have been studied to describe the
boundary of the compactification of $\Mb_3$ and there is a rich literature on
this subject \cite{artebani, glass, hassett, hyeon, kang, kondo, mum-ens,
  yukie}. Surprisingly, the mentioned problem has never been fully addressed
before (still, see \cite[Prop.14 and Prop.15]{howe00} in the case of Ciani
quartics). Known results about the relations between plane quartics and
hyperelliptic genus 3 curves are basically contained in the following
proposition \cite[p.156]{clemens}, \cite{harriscm},
\cite[p.134]{harris-moduli}.

\begin{proposition}\label{2->1} Suppose $p \ne 2$.
  Let $s >0$ be an integer, $G \in \O[x_1,x_2,x_3]$ be a primitive (\ie~the
  gcd of its coefficients is 1) quartic form and $Q \in \O[x_1,x_2,x_3]$ be a
  primitive quadratic form. Let us assume that $\bar{Q}$ is irreducible and
  that $\bar{Q}=0$ intersects $\bar{G}=0$ transversely in $8$ distinct
  $\overline{k}$-points.  Let us denote $\Cc/\O$ the subscheme of the weighted
  projective space $\P^{(1,1,1,2)}$ defined by
  \begin{equation} \label{eq:modelhyp}
    \begin{cases}
      y^2 + G = 0, & \\
      \pi^s y - Q =0. &
    \end{cases}
  \end{equation}
  Then the generic fiber is isomorphic to the plane smooth quartic
  $C/K : Q^2 +\pi^{2s} G=0$ which has good hyperelliptic reduction. The
  special fiber of $\Cc$ is isomorphic to the double cover of $\bar{Q}=0$
  ramified over the 8 distinct intersection $\overline{k}$-points of
  $\bar{Q}=0$ with $\bar{G}=0$.
\end{proposition}
This motivates the following definitions.

\begin{definition}\label{def:HGRM}
  Let $C/K$ be a plane smooth quartic. When $p\ne 2$, we say that $C$ admits a
  \emph{\PHM} if there exist an even integer $s>0$, a primitive quartic form
  $G \in \O[x_1,x_2,x_3]$ and a primitive quadric $Q \in \O[x_1,x_2,x_3]$ with
  $\bar{Q}$ irreducible such that $Q^2+\pi^{2s} G=0$ is $K$-isomorphic to
  $C$. If, moreover, $\bar{Q}=0$ intersects $\bar{G}=0$ transversely in $8$
  distinct $\overline{k}$-points, the model $Q^2+\pi^{2s} G=0$ is a \emph{good
    \PHM} of $C$.  We refer to Theorem~\ref{th:is2} for the definition when
  $p=2$.
\end{definition}

When $C/K$ admits a plane quartic model over $\O$ which special fiber is
smooth then obviously this plane model is a stable model of $C$. When $C$
admits a good toggle model, it has good hyperelliptic reduction and
\eqref{eq:modelhyp} is a stable (actually smooth) model of $C$. The converse
holds as well, this is the main result of Section~\ref{sec:Liu}.
\begin{theorem}[See Theorems~\ref{th:goodmodel-f} and
  \ref{th:is2}] \label{th:goodmodel} Let $K$ be a discrete valuation field
  with residue field of characteristic $p \geq 0$. Let $C/K$ be a smooth plane
  quartic.  Then $C$ admits good hyperelliptic reduction over $K$ if and only
  if $C$ has a good {\PHM} over $K$.
\end{theorem}
By using this result, the methods of \cite{kollar} and \cite{elsenhans-stoll}
and some ingredients from invariant theory, the question of the reduction type
(good quartic, good hyperelliptic or bad) over $K$ of a smooth quartic may be
answered (see remark~\ref{rem:kollar}).  \medskip

We turn now to the determination of the type \emph{when one allows extensions
  of $K$}. In the spirit of the work of \cite{liug2} on genus 2, we aim at
getting a characterization of the potential reduction type in terms of the
valuations of certain invariants of $C/K$ under the action of
$\SL_3(\bar{K})$. Let us start by recalling some results on the homogeneous
algebra of invariants of ternary quartic forms.

In \cite{dixmier}, Dixmier worked out a list of 7 homogeneous polynomial
invariants for the equivalence of ternary quartic forms under the action of
$\SL_3(\C)$, denoted $I_3, I_6, I_9,I_{12},I_{15},I_{18}$ and $I_{27}$. They
form a Homogeneous System Of Parameters (HSOP) for the algebra of invariants
$\mathcal A$ over $\C$, which means that the radical of the ideal they
generate is equal to the irrelevant ideal of the homogeneous algebra
$\mathcal A$ (see Definition~\ref{def:hsop} and remark~\ref{rem:hsop2}).
Later, this list is completed by Ohno (\cite[Theorem 4.1]{ohno}, see also
\cite{elsenhans}) into a list of 13 homogeneous generators of $\mathcal A$ as
$\C$-algebra. Theses invariants are defined over $\Z[\frac{1}{2 \cdot 3}]$,
and are now called the \emph{Dixmier-Ohno invariants}.  We denote $\DOv$ the
list of these invariants, $\DOv(F) \in K^{13}$ their values at a ternary
quartic form $F$ and when at least one of these values is not zero, we denote
$\DO(F)$ the corresponding point in the weighted projective space with weights
$3,6,9,9,12,12,15,15,18,18,21,21,27$.  They determine the isomorphism classes
of smooth quartic curves over $\C$ in the sense that two smooth quartics
$C_i:\,F_i=0$ are isomorphic over $\C$ if and only if $\DO(F_1)=\DO(F_2)$.

This result over $\C$ is clearly not sufficient for our purposes as we want to
deal with fields of arbitrary characteristic and with reduction
properties. Fortunately, Seshadri in \cite{Se} has generalized the Geometric
Invariant Theory (GIT) of Mumford over an arbitrary ring $R$. For the
convenience of the reader, we include in Section~\ref{sec:appendix} some
general results on this topic and proofs we were unable to find in the
literature. We also prove, under some assumptions on $R$\footnote{the
  assumptions are satisfied for instance for local rings and rings of integers
  of global fields.}, that there exists a HSOP over $R$, \ie~a finite list of
invariants defined over $R$ being a HSOP after base change to all residue
fields of $R$. It is a bit of a surprise as one knows that there exist
invariants over fields of small characteristics which do not come from
reducing invariants in characteristic $0$ (see \cite{basson} for binary forms
and remark~\ref{rem:nonlift} for ternary quartics).  In principle, one should
be able to exhibit a HSOP over $\Z$, but in practice the degrees may be too
large to be practical. This is why in Theorem~\ref{th:domodp}, we determine
explicit HSOP over $\O$ for ternary quartic forms which depend only on the
characteristic of the residue field. This is achieved in all characteristics
but $3$, for which we only have a conjectural HSOP.  (see
Theorem~\ref{th:domod3}). Note also that even for some large characteristics,
like $523$, the Dixmier invariants alone may not be a HSOP unlike in
characteristic $0$.  With the normalization of \cite[p.426]{gelfand},
\cite{demazure}, the \emph{discriminant} $D_{27}$ of a ternary quartic is a
degree $27$ polynomial with coefficients in $\Z$ and always belongs to our
HSOP (up to a multiplicative unit). It is related to the Dixmier invariant
$I_{27}$ by $2^{40} \cdot I_{27}=D_{27}$. A plane quartic $F=0$ over any field
is smooth if and only if $D_{27}(F) \ne 0$.

As we want to deal with valuations of the invariants, we need to introduce the following notation before stating our results.

\begin{definition} \label{def:minimalrepresentative} Let
  $\dv=(d_0,\ldots,d_n)\in\Z_{>0}^{n+1}$.  Given an element
  $\xv=(x_0,\ldots,x_n) \in K^{n+1} \setminus \{0\}$, we denote by
  $x=(x_0:\ldots:x_n)$ the corresponding point in the weighted projective
  space $\P^{\dv}(K)$. After a possible finite extension of $K$, one can
  always find a representative $(X_0,\ldots,X_n) \in \O^{n+1}$ of $x$ such
  that one of the $X_i$'s has valuation $0$, so it represents a section in
  $ \P^{\dv}(\O)$. We call such a representative a \emph{minimal
    representative} and denote it $\xv^{\min}$.
\end{definition}
A minimal representative is not uniquely defined but
two minimal representatives differ by the action of a unit and their
valuations are coordinate-wise equal.

\begin{definition} \label{def:normalizedval}
If $y \in K$ and $e \in \Z_{>0}$, we call the \emph{normalized
valuation of degree $e$ of $y$ with respect to $\xv$} the
(possibly infinite) number
$$v_x(y) = v(y)/e - \min \{v(y)/e,\,v(x_i)/d_i:\,i=0,\dots,n\} \in
\Q_{\geq 0} \cup \{\infty\},$$ (By convention $v(0)=\infty$).  We have
$v_x(y)=0$ if and only if $y\ne 0$ and $x_i^{e}y^{-d_i}\in \O$ for all
$i\le n$.

Notice that if $(X_0,\ldots,X_n,Y)$ is a minimal representative of the
point $(x_0:\ldots:x_n:y) \in \P^{\dv,e}$, then $v_{x}(y)=v(Y)/e$.
\end{definition}

In the context of invariants, we will use the following definition and notation.

\begin{definition} \label{def:minimalform} Let $ \Iv =(I_0, \ldots, I_n)$ be a
  list of homogeneous polynomials of degree $\underline{d}$ in
  $K[T_0,\ldots, T_m]$. For $F \in K^{m+1}\setminus \{0\}$, if
  $\Iv(F) = (I_0(F), \ldots, I_n(F))$ is not the zero vector, and if $J$ is a
  homogeneous polynomials of degree $e$ we denote by $v_{I}(J(F))$ the
  normalized valuation of degree $e$ of $J(F)$ with respect to $\Iv(F)$. It
  only depends on the point of $\P^{\underline{d}}(K)$ defined by $F$.
  Finally, if $F_0 \in \O^{m+1}$ and $\Iv(F_0) = \Iv(F_0)^{\min}$ we say that
  $F_0$ is a \emph{minimal form}.
\end{definition}

\begin{remark}
  Let $J, I_0, \dots, I_n\in E\subseteq \O[T_0, \dots, T_m]$ be homogeneous
  elements of some graded sub-$\O$-algebra and suppose that
  \begin{equation} \label{eq:hsg}
    E_+\subseteq \sqrt{(I_0, \dots, I_n)}
  \end{equation}
  that is, $\{I_0,\dots, I_n\}$ is a homogeneous system of radical generators
  (HSORG) of $E$, see Definition~\ref{def:hsop}. In particular, HSOP are HSORG
  with further condition on the cardinality. The condition $v_I(J(F))=0$ is
  equivalent to $F\in D_+(J)$, the principal open subset of $\Proj E$
  associated to $J$. Here $F$ is viewed as a point of $(\Proj E)(K)$. In
  particular, the condition $v_{I}(J(F))=0$ does not depend on the choice of
  the system $I_0,\dots, I_n$ as long as the latter is a HSORG of $E$.
\end{remark}

Considering ternary quartic forms $F$, a natural question, and a key argument
for the sequel, is to know if, after a possible finite extension of $K$, there
exists a minimal form $\GL_3(K)$-equivalent to $F$. We show in
Proposition~\ref{prop:ssred} that this is the case when $F=0$ is GIT-stable
over $K$ and the list $\Iv$ is a HSORG over $\O$.  Using this result, we
obtain the following theorem.

\begin{theorem}[See Theorem~\ref{th:gnhr}] \label{th:GNHR} Let $K$ be a
  discrete valuation field with valuation $v$, valuation ring $\O$ and residue
  field $k$ of characteristic $p \geq 0$.  Let $\Iv$ be a list of invariants
  giving rise to a HSORG for ternary quartic forms under the action of
  $\SL_{3,\O}$. Then a smooth plane quartic $C/K : F=0$ has potentially good
  quartic reduction if and only if $v_{I}(D_{27}(F))=0$ (in other words, for
  any invariant $I_i$ of degree $d_i$ in the list $\Iv$, we have
  $I_i(F)^{27}/D_{27}(F)^{d_i}\in \O$).
\end{theorem}
This theorem is proved in Section~\ref{sec:Shah}. Once worked out explicit
HSOP in Section~\ref{sec:dixm-invar-char}, it gives a practical criterion in
all characteristics. Indeed, notice, that though we use the existence of a
minimal form, it is not necessary to compute it to check whether
$v_{I}(D_{27}(F))=0$. Similarly, we do not need the minimal form to compute an
equation of the special fiber when the curve has potentially good quartic
reduction, at least generically and if $p>7$ (see remark~\ref{rmk:nhspecial}).
Some illustrations are given in Example~\ref{ex:cm} and in
Theorems~\ref{ex:picard}, \ref{ex:picard2} and \ref{ex:picard3} with the
characterization of potentially good quartic reduction for Picard curves in
terms of low degree polynomials in their coefficients when $\Char(K) \ne 3$.

From Section~\ref{sec:comphyp} and beyond, we restrict to a discrete valuation
ring $\O$ with residue field $k$ of characteristic $p=0$ or $p>7$ (see a
discussion about this restriction in remark~\ref{rem:p>7}).  We focus there on
distinguishing between curves which have potentially good hyperelliptic
reduction or have not potentially good reduction. In
Proposition~\ref{prop:MHR}, we first prove that $C/K : F=0$ potentially admits
a (not necessarily good) \PHM{} if and only if any minimal representative of
$\DO(F)$ reduces modulo $\pi$ to $\DO(Q_0^2)$ where $Q_0=x_2^2-4 x_1 x_3$ is a
fixed non-degenerate quadric (the choice of this particular one will become
clearer later).

It remains to characterize in terms of the invariants when a \PHM{}
$Q_0^2+\pi^{2s} G=0$ is good, \ie~when the intersection of $\bar{Q}_0=0$ with
$\bar{G}=0$ is transversal. This is the content of
Section~\ref{sec:char-hyper-reduct}. In order to do so, we introduce a
well-known $\SL_2(\bar{K})$-equivariant linear morphism $b_8$ from the space
of ternary quartic forms $F$ to the space of binary octic forms.  The action
on the ternary quartics is given through a morphism
$h : \SL_2(\bar{K}) \to \SL_3(\bar{K})$ whose image denoted $\SO_3(\bar{K})$
is the stabilizer of $Q_0$.  Starting from a good \PHM{} $F=Q_0^2+\pi^{2s} G$
of $C/K$, a quick computation shows that $b_8(F)=\pi^{2s} b_8(G)$ and that
$y^2=\overline{b_8(G)}$ is an affine equation of the special fiber of the
stable model. The algebra of invariants of binary octics under the action of
$\SL_2(\bar{K})$ is generated by the classical Shioda invariants
$j_2,\ldots,j_{10}$, see \cite{shioda67}. The invariants $j_2,\ldots,j_7$
still form a HSOP over $\O$, see Corollary~\ref{cor:Shiodas} and the
discussion before. Evaluating the Dixmier-Ohno invariants on a \PHM{}
$Q_0^2+\pi^{2s} G=0$ and the Shioda invariants $j_2,\ldots,j_7$ on $b_8(G)$
gives algebraic relations between the two sets and considering their
$\pi$-expansion reveals a list
$\iota=(\iota_6,\iota_9,\iota_{12},\iota_{15},\iota_{18},\iota_{21})$ of six
explicit invariants for ternary quartic forms, see
equation~\eqref{eq:DOshort}. Moreover we get a similar
congruence~\eqref{eq:trompe} between the discriminant of binary octics,
$D_{14}$, and $I_3^5\, I_{27}$.  For a good \PHM{}, one has
$v_{\iota}(I_3^5(F)\, I_{27}(F)) = v(D_{14}(b_8(G)))=0$. Actually, we have the
following equivalence.

\begin{theorem} \label{th:main} Let $K$ be a discrete valuation field with
  valuation $v$, valuation ring $\O$ and a uniformizer $\pi$. Let
  $k=\Oc/\langle \pi \rangle$ be the residue field of characteristic
  $p \ne 2,3,5$ and $7$.  Let $C/K$ be a smooth plane quartic defined by
  $F=0$. Then $C$ has potentially good hyperelliptic reduction if and only if
         $$v_{\DO}(I_3(F))=0, \quad v_{\DO}(I_{27}(F))>0 \; \textrm{and} \;
         v_{\iv}(\,I_3(F)^5\, I_{27}(F)\,)=0.$$ Under these conditions, one
         can also obtain an explicit equation for the special fiber.
       \end{theorem}

To prove that the conditions of Theorem~\ref{th:main} are also sufficient, one
first observes that they imply the existence of a \PHM{} $F=Q_0^2+\pi^{2s} G$
for $C$. Then, a crucial step is to prove that we can choose it such that
$\overline{b_8(G)}$ is GIT-semi-stable. In order to do this, we
rigidify a \PHM{} in the following way: there is, up to a multiplicative
constant, a unique contravariant $\rho$ of order 2 and degree 4, which can be
seen as a quadratic form whose coefficients are degree $4$ polynomials in the
coefficients of $F$ (see \cite{dixmier} and \cite{LRS16}). After a possible finite extension of $K$ and completion of $K$, one can always find
a \PHM{} $F$ such that $\rho(F) = \rho(Q_0^2)$. This extra-information allows us to control simultaneously the integrality of the coefficients of the \PHM{} and of  the form $b_8(G)$ we are looking for (see
Lemma~\ref{a>=b}).\\

Let us point out that the conditions in Theorem~\ref{th:main} may give the
wrong idea that the study of the valuations of $I_{3}$ and $I_{27}$ would be
enough to characterize the reduction. But this is not true since we need to
consider them inside the broader lists of invariants $\DO$ and $\iota$, and
normalize them accordingly to the valuations of all the invariants of these
lists.  Notice that a similar phenomena occurs for Siegel modular forms. In
\cite{LR19}, we show that the pull-back of the Siegel modular form $\chi_{28}$
to the space of quartics is $-2^{171} \cdot 3^3 \cdot I_{27}^3 I_3$. Now, over
$\C$, the non-vanishing of this modular form together with the vanishing of
the modular form $\chi_{18}$ (which is known to be related to $I_{27}^2$)
characterizes the locus of hyperelliptic Jacobians among abelian varieties
(see \cite[p.847]{tsuyumine-g3}). A natural guess is to expect that they would
be the key forms to find the reduction type of a quartic. However, their
expressions in terms of invariants seem to lead to the absurd statement that
the vanishing of $\chi_{18}$ always implies the vanishing of $\chi_{28}$.
Within this article, one sees that this is because these forms would have to
be considered inside broader lists to become significant.

One of the motivations of this work was, in the continuation of \cite{WINE1},
\cite{WINE1-2} and \cite{KLLRSS17}, to understand the primes dividing the
discriminant $I_{27}$ of a smooth plane quartic with complex multiplication
over a number field only in terms of the CM-field. We hope that the present
work can be useful for this task. Another direction is to refine the study of
the bad reduction type in the spirit of \cite{BW}, \cite{maistret} and to
describe more precisely the degenerated cases. We only have a quick peek at
this question in remarks~\ref{rk:singular} and \ref{rk:singular2} and set
aside an exhaustive study for a forthcoming article.

\section{Hyperelliptic reduction in terms of \PHM s} \label{sec:Liu}

\newcommand{\caract}{\mathrm{char}}
\newcommand{\cO}{{\mathcal O}}
\newcommand{\cC}{{\mathcal C}}
\newcommand{\W}{{\mathcal W}}
\newcommand{\cZ}{{\mathcal Z}}
\newcommand{\ol}{\overline}

In this section $K$ is a discrete valuation field with valuation $v$,
valuation ring $\O$ and a uniformizer $\pi$. Let $k=\Oc/\langle \pi \rangle$
be the residue field of characteristic $p \geq 0$.  Before proving
Theorem~\ref{th:goodmodel}, we need some preliminary results.

\subsection{Some general facts} \label{general-facts}

In this section we prove the existence of suitable models when $C$ has good
hyperelliptic reduction over $K$ (see Theorems \ref{th:goodmodel-f} and
\ref{th:is2}).

We start with some geometric observations. Suppose $C$ has good reduction over
$K$. Let $\cC$ be a smooth projective model over $\cO$.  Let
$\omega_{\cC/\cO}$ be the dualizing sheaf ($=\Omega^1_{\cC/\cO}$ as $\cC$ is
smooth). Then $\omega_{\cC/\cO}$ is a base-point free invertible sheaf and
induces a morphism
$$ f: \cC \to \P(H^0(\cC, \omega_{\cC/\cO}))\simeq \P^2_\cO. $$
On the generic fiber, $f$ is a closed immersion because $C$ is a plane
quartic. Let $\cZ=f(\cC)$ be the image of $f$ endowed with the reduced
subscheme structure. Then $f: \cC\to \cZ$ is a birational finite morphism. As
$\cC$ is normal, $f$ is just the normalization morphism.

\begin{fact} \label{fact} If we have a finite birational morphism
  ${\W}\to \cZ$, and if ${\W}$ is normal, then $\W \to \cZ$ is the
  normalization morphism (\cite[Proposition 4.1.22]{liu-book}), hence ${\W}$
  is isomorphic to $\cC$.%
\end{fact}

If $\omega_{\cC/\cO}$ is very ample, then $f$ is a closed immersion and $\cC$
is a smooth quartic in $\P^2_\cO$. Suppose from now on that $\omega_{\cC/\cO}$
is not very ample. Then $f_k : \cC_k \to \P^2_{k}$ is a finite generically
\'etale morphism of degree $2$ into its image $(\cZ_k)_{\mathrm{red}}$ which
is a smooth conic over $k$.  The scheme $\cZ$ is defined by an irreducible
quartic form $F\in \cO[x_1, x_2,x_3]$ as
\begin{equation} \label{eq:defZ}%
  \cZ=V_+(F)\subset \P^2_\cO.
\end{equation}
Let $q\in k[x_1,x_2,x_3]$ be a nondegenerate quadratic form defining
$(\cZ)_{\mathrm{red}}$:
\begin{equation}
  \label{eq:conic}
  \cZ_{\mathrm{red}}=V_+(q) \subset \P^2_k.
\end{equation}
In the sequel, $q$ will be fixed once for all.  As $V_+(\ol{F})=\cZ_k$, the
reduction $\overline{F}$ is equal to a multiple of $q^2$. Scaling $F$ by a
unit of $\cO$, we can suppose that
\begin{equation}
  \label{eq:Fq}
  \ol{F}=q^2.
\end{equation}
Given such an $F$, equation~\eqref{eq:Fq} implies that there exist a positive
integer $r>0$, a quadratic form $Q \in \cO[x_1,x_2,x_3]$ and a quartic form
$G\in \cO[x_1,x_2,x_3]$ such that
\begin{equation} \label{eq:F-Q2}%
  F=Q^2+\pi^{r} G,
\end{equation}
with $\overline{Q}=q$ and $\overline{G}\ne 0$. Note that if $\hat{\cO}$ is the
completion of $\cO$, then $F$ is not a square in $\hat{\cO}[x_1,x_2,x_3]$
because $C$ is geometrically reduced. So the biggest possible power $\pi^r$
dividing $F-Q^2$ in $\cO[x_1,x_2,x_3]$ for all liftings $Q$ of $q$ exists.
\medskip

We endow the polynomial ring $\cO[x_1,x_2,x_3,y]$ with the grading
$\deg x_i=1$ and $\deg y=2$.  The following lemma will be repeatedly used:

\begin{lemma}\label{normality-proj} Let $B$ be a graded $\cO$-algebra of the following form
  $$ B= \cO[x_1,x_2, x_3, y]/(Q-\pi^s y, y^2+\pi^{\epsilon} T), \quad T\in \cO[x_1,x_2,x_3]+\cO[x_1, x_2, x_3]y$$
  with $s\ge 0$, $\epsilon\in \{0, 1\}$ and $2s+\epsilon>0$.  Then the following
  properties hold:
  \begin{enumerate}[label=\emph{(\alph*)}]
  \item $B$ is flat over $\cO$.
  \item Suppose that $B\otimes K$ is normal.  If $B\otimes k$ is reduced when
    $\epsilon=0$ and $T\notin (\pi, Q, y)$ when $\epsilon=1$, then $B$ is
    normal.
  \item If furthermore, $F\in (y^2+\pi^{\epsilon}T)K^*$ in
    $K[x_1, x_2, x_3,y]/(Q-\pi^s y)$, then $\Proj B\simeq \cC$.
  \end{enumerate}
\end{lemma}

\begin{proof} (a) The flatness is equivalent to $B$ being
  $\pi$-torsion-free. The latter follows from a straightforward computation.

  (b) As $\dim B\otimes K=\dim B\otimes k=2$, $\Spec B$ is equidimensional. It
  is then easy to see that $B$ is a local complete intersection
  (\cite{liu-book}, Exercise 6.3.4(c)).  By Serre's criterion
  (\cite{liu-book}, Corollary 8.2.24), it is then enough to check that $B_\p$
  is normal for height $1$ prime ideals $\p$ of $B$. This follows from the
  hypothesis on $B\otimes K$ if $\pi\notin\p$.

  Suppose that $\pi\in \p$. Then $\p$ is a minimal prime ideal over $\pi
  B$. Let us show that $B_\p$ is a discrete valuation ring, or equivalently,
  that $\p B_\p$ is principal.  If $\epsilon=0$, as $B\otimes k$ is reduced,
  $\p B_\p$ is generated by $\pi$.  If $\epsilon=1$, then $\p=(\pi, y)$ is the
  unique prime ideal over $\pi B$.  This prime ideal is the image in $B$ of
  the prime ideal $(\pi, Q, y)$ of $\cO[x_1, x_2, x_3,y]$.  By hypothesis,
  $T\notin \p$ in $B$, therefore $\p B_\p$ is generated by $y$.

  (c) The canonical map $\cO[x_1, x_2, x_3] \to B$ is finite. The hypothesis
  on $F$ easily implies that $F$ belongs to the ideal
  $(Q-\pi^sy, y^2+\pi^{\epsilon}T)$ in $\cO[x_1, x_2, x_3, y]$.  So we get a
  canonical map $\phi : \cO[x_1,x_2,x_3]/(F)\to B$ which is an isomorphism
  after tensoring with $K$, thus $\phi$ is finite and injective and the
  corresponding morphism $\Proj B\to \cZ$ is finite and birational. We
  conclude by Fact~\ref{fact}.
\end{proof}

\begin{remark}\label{rk:singular} When $C$ has irreducible stable (not
  necessarily smooth) reduction over $K$, the situation is very similar to the
  good reduction case. Let $\cC$ be the stable model of $C$ over $\cO$.  Then
  $\omega_{\cC/\cO}$ is still base-point free and induces a morphism
$$ f: \cC \to \P(H^0(\cC, \omega_{\cC/\cO}))\simeq \P^2_\cO. $$
The reduced image $\cZ=f(\cC)$ is a quartic and $f: \cC\to \cZ$ is a
birational finite morphism, equal to the normalization morphism. If
$\omega_{\cC/\cO}$ is very ample, then $\cC\simeq \cZ$ is a stable quartic in
$\P^2_\cO$.  If $\omega_{\cC/\cO}$ is not very ample, then one can show that
$(\cZ_k)_{\mathrm{red}}$ is a smooth conic over $k$ and
$f_k : \cC_k \to (\cZ_k)_{\mathrm{red}}$ is a finite generically \'etale
morphism of degree $2$. The reduction is a singular hyperelliptic curve.
\end{remark}

Next we examine when the special fiber of $\Proj B$ is reduced or semi-stable
(in the Deligne-Mumford sense).  A point in a scheme over $k$ is said to be
\emph{separable over $k$} if its residue field is a separable extension of
$k$.

\begin{lemma} \label{lm:reduced} Suppose $\caract(k)\ne 2$. Let
  $g\in k[x_1, x_2, x_3]$ be a quartic form and let
$$W=\Proj (k[x_1,x_2,x_3,y]/(q, y^2+g))$$
with $\deg x_i=1$ and $\deg y=2$. Then the following properties are true:
\begin{enumerate}[label=\emph{(\alph*)}]
 \item  $W$ is reduced if and only if $g\not\equiv 0 \mod q$;
 \item   $W$ is semi-stable (respectively smooth) if and only if the curves
 $q=0$ and $g=0$ in $\P^2_k$ intersect at separable points over $k$,
 with multiplicity at most $2$ (respectively equal to $1$).
\end{enumerate}
\end{lemma}

\begin{proof} (a) The condition is clearly necessary. Suppose that the
  condition is satisfied. Let us prove that $B_k:=k[x_1,x_2,x_3,y]/(q, y^2+g)$
  is reduced. Let $A=k[x_1,x_2,x_3]/(q)$. This is an integrally closed domain
  and we have $B_k=A\oplus Ay$.  By hypothesis $g\ne 0$ in $A$. So
  $\Frac(A)[y]/(y^2+g)$ is separable over $\Frac(A)$. Hence,
  $$B_k=A\oplus Ay\subset \Frac(A)\oplus \Frac(A)y=\Frac(A)[y]/(y^2+g)$$ is reduced.

  (b) The natural projection $\phi: W\to V_+(q)$ is finite and \'etale away
  from the subscheme $D:=V_+(q, g)$. Let $w_0\in W$ be a point lying over some
  point of $D$. Then $\phi$ is ramified at $w_0$. Let us show that $w_0$ is a
  separable point over $k$. When $W$ is smooth at $w_0$, this is proved for
  instance in \cite{LiLo}, Lemma 3.3. The case when $W$ is singular at $w_0$
  is proved in \cite{liu-book}, Proposition 10.3.7(b).

  For the rest of (b), we can suppose that $k$ is algebraically closed.  A
  local equation of $W$ at a point lying over an intersection point $z\in D$
  is $y^2=x^ra$ where $r$ is the intersection number at $z$ and
  $a\in \cO_{V_+(q), z}^{*}$. The condition for $W$ to be semi-stable or
  smooth is then clear.
\end{proof}

\begin{lemma} \label{lm:reduced2} Let $\caract(k)=2$. Let
  $h, g\in k[x_1, x_2, x_3]$ be respectively a quadratic and a quartic
  forms. Let
  $$B_k=k[x_1,x_2,x_3,y]/(q, y^2+hy+g)$$
  and let $W=\Proj B_k$ with $\deg x_i=1$ and $\deg y=2$. Then the following
  properties are true:
  \begin{enumerate}[label=\emph{(\alph*)}]
  \item $W$ is reduced if and only if $h$ is prime to $q$ or, $h$ is divisible
    by $q$ and $g$ is not a square mod $q$;
  \item if $h$ is divisible by $q$, then $W$ is not semi-stable;
  \item if the curves $V_+(q)$ and $V_+(h)$ intersect transversely at
    separable points of $\P^2_k$, then $W$ is semi-stable.
  \end{enumerate}
\end{lemma}

\begin{proof}
  (a) The condition is clearly necessary. Let us show the converse. Let $A$ be
  defined as in the proof of Lemma~\ref{lm:reduced}(a).  If $h$ is prime to
  $q$, then the proof is similar to that of Lemma~\ref{lm:reduced}(a).
  Suppose now that $h \equiv 0 \pmod{q}$. If $b=a_0+a_1y\ne 0$ is nilpotent in
  $B_k$, then so is $b^2=a_0^2+a_1^2g\in A$. Therefore $a_0^2+a_1^2g=0$ in
  $A$. As the latter is integrally closed, $a_0/a_1\in A$ and $g$ is a square
  in $A$, contradiction.

  (b) Denote by $\phi: W\to Z:=V_+(q)$ the natural projection. If $h$ is
  divisible by $q$, then $\phi$ is purely inseparable, so the singularities of
  $W$ are all cuspidal. On the other hand, $W$ is not smooth as its arithmetic
  genus is $3$ while its geometric genus is $0$.  So $W$ is not semi-stable.

  (c) We can suppose $k$ is separably closed. The above cover $W\to V_+(q)$ is
  \'etale away from $D$. So $W\setminus \phi^{-1}(D)$ is smooth. The local
  equation of $W$ at a point lying over $z\in D$ is $y^2+tu(t)y+g(t)$ where
  $t$ is a local parameter of $Z\simeq \P^1_k$ at $z$, with $u(0)\ne 0$. After
  a translation of $y$ by some constant in $\ol{k}$, the local equation has
  the shape $y^2+tu(t)y+t^rf(t)$ with $r\ge 1$ and $f(0)\ne 0$. If $r=1$, the
  point is smooth, otherwise it is a node.
\end{proof}

\begin{remark}
  The condition in (c) is not necessary. We can have inseparable
  non-transverse intersection points with $W$ smooth.
\end{remark}

\begin{remark} The smoothness condition is less easy to express than in the
  $\caract(k)\ne 2$ case. Keep the notation of the proof of Part (b). Let
  $w_0\in W$. We know that $w_0$ is a smooth point if
  $w_0\notin \phi^{-1}(D)$. Suppose that $\phi(w_0)=z_0\in D$. The Jacobian
  criterion implies that $W$ is smooth at $w_0$ if and only if the matrix
  $$
  \left(
    \begin{matrix}
      \partial q/\partial x_1\hfill && \partial q/\partial x_2 \hfill && \partial
      q/\partial x_3 \\
      y_0\partial h/\partial x_1+\partial g/\partial x_1 &&
      y_0\partial h/\partial x_2+\partial g/\partial x_2 &&
      y_0\partial h/\partial x_3+\partial g/\partial x_3
    \end{matrix}
  \right),
  $$
  where $y_0\in \bar{k}$ is a square root of $g(z_0)$, has rank 2 at
  $z_0$.
\end{remark}

\subsection{Proof of Theorem~\ref{th:goodmodel}}

In this subsection we prove the following result.

\begin{theorem} \label{th:goodmodel-f} Let $K$ be a discrete valuation field
  with valuation $v$, valuation ring $\O$ and a uniformizer $\pi$. Let
  $k=\Oc/\langle \pi \rangle$ be the residue field of characteristic
  $p \ne 2$. Suppose that the smooth plane quartic $C$ has good hyperelliptic
  reduction over $K$.  Fix $q$ and a defining equation $F$ of $C$ as in
  formulas~\eqref{eq:conic} and \eqref{eq:Fq}. Then there exist homogeneous
  polynomials $Q, \ G\in \cO[x_1,x_2,x_3]$ of degrees $2, 4$ respectively, and
  a positive integer $s$ such that
  \begin{enumerate}[label=\emph{(\alph*)}]
  \item $F=Q^2+\pi^{2s} G$ and $\ol{G}\ne 0$;
  \item The smooth projective model  $\cC$ of $C$ over $\cO$ is
    $$\Proj \left(\cO[x_1, x_2, x_3, y]/(Q-\pi^s y, y^2+G) \right)$$
    in the weighted projective space $\P^{(1,1,1,2)}_{\cO}$.
  \item The conic $\ol{Q}=0$ and the quartic $\overline{G}=0$ in $\P^2_k$
    intersect transversely at an effective divisor of degree $8$ over $k$;
  \end{enumerate}
\end{theorem}

\begin{proof} Let $Q, G$ be as in equation~\eqref{eq:F-Q2} with the biggest
  possible $r$.  We claim that $\overline{G}$ is prime to $q$. Otherwise,
  $G=QR+\pi G_1$ with $R, G_1$ homogeneous of degrees $2$ and $4$
  respectively.  Then the polynomial $F-(Q+\pi^rR/2)^2$ is divisible by
  $\pi^{r+1}$. Contradiction with the assumption on $r$.

  We now prove the theorem. Let $s=\lfloor r/2 \rfloor$ and
  $\epsilon=r-2s\in \{0,1\}$. Let
  $$\W=\Proj\left(\cO[x_1, x_2, x_3, y]/(Q-\pi^s y, y^2+\pi^{\epsilon} G)\right).$$
  If $\epsilon=1$, as $G\notin (\pi, Q,y)$, Lemma~\ref{normality-proj} implies
  that $\W$ is normal and isomorphic to $\cC$, hence smooth. This is absurd as
  the special fiber of $\W$ is non-reduced. So $\epsilon=0$.  The special
  fiber $\mathcal W_k$ is reduced by Lemma~\ref{lm:reduced} (a). Again by
  Lemma~\ref{normality-proj}(c), ${\W}\simeq \cC$.  Part (c) is just
  Lemma~\ref{lm:reduced}(b).
\end{proof}

\subsection{Good hyperelliptic reduction in $\caract(k)=2$}
When the characteristic of $k$ is equal to $2$, \PHM s have singular close
fibers. However, we still have an analog result to Theorem~\ref{th:goodmodel}
in this case.

\begin{theorem} \label{th:is2} Let $K$ be a discrete valuation field with
  valuation $v$, valuation ring $\O$ and a uniformizer $\pi$. Let
  $k=\Oc/\langle \pi \rangle$ be the residue field of characteristic $p =
  2$. Suppose that the smooth plane quartic $C$ has good hyperelliptic
  reduction over $K$.  Fix $q$ and a definition equation $F$ of $C$ as in
  formulas~\eqref{eq:conic} and \eqref{eq:Fq}.  Then there exist homogeneous
  polynomials
  $$Q, H, G\in \cO[x_1,x_2,x_3]$$
  of degrees $2, 2, 4$ respectively, $u\in 1+\pi \cO$ and a positive integer
  $s$ such that
  \begin{enumerate}[label=\emph{(\alph*)}]
  \item $uF=Q^2+\pi^{s} QH+\pi^{2s}G$ and $\ol{H}$, $\ol{G}$ are prime to $q$;
  \item The smooth projective model $\cC$ of $C$ over $\cO$ is
    $$\Proj \left(\cO[x_1, x_2, x_3, y]/(Q-\pi^s y, y^2+Hy+G) \right)$$
    in the weighted projective space $\P^{(1,1,1,2)}_{\cO}$.
  \end{enumerate}
\end{theorem}

\begin{proof} In all the proof, capitalized italic letters denote homogeneous
  polynomials in $\cO[x_1, x_2, x_3]$ of degree $2$ or $4$ (the exact degree
  will be clear from the context).  \medskip

  {\bf Step 1.}  {\sl Suppose  we have a decomposition
    $$F=Q^2+\pi^s QH+\pi^r P$$ with $1\le r< 2s$ (including  $s=+\infty$, in
    which case $\pi^s$ stands for $0$). Then we have a new decomposition
    $$F=Q_1^2+\pi^{s_1} Q_1H_1+ \pi^{r+1}P_1$$
    with $Q_1\in Q+\pi^{[r/2]}\cO[x_1,x_2, x_3]$ and
    $s_1\ge \min\{ s, r, v(2)+r/2\}$, where $v(2)$ is the normalized
    $K$-valuation of $2$.}

  Let $\ell=[r/2]$ and $\epsilon=r-2\ell$. Consider
  $$ \W=\Proj\left(\cO[x_1, x_2, x_3,y]/(Q-\pi^{\ell}y, y^2+\pi^{s-\ell}Hy+\pi^{\epsilon}P)\right).$$

  \underline{\bf Case 1}. Suppose that $\epsilon=1$ ($\ell\ge 0$).  The special
  fiber of $\W$ is not reduced (Lemma~\ref{lm:reduced2}(a) if $\ell>0$), by
  Lemma~\ref{normality-proj} (b) and (c), we must have
  $$P\in (\pi, Q_0, y)\cap \cO[x_1, x_2, x_3]=(\pi, Q)\cO[x_1,x_2, x_3].$$
  Write $P=QH_0+\pi P_1$, then
  $$F=Q^2+(\pi^s H+\pi^{r}H_0)Q+\pi^{r+1}P_1.$$

  \underline{\bf Case 2}. \ Suppose that $\epsilon=0$. As $s>\ell$, $\W_k$ is not semi-stable by Lemma~\ref{lm:reduced2}(b). Similarly to Case 1, this implies that
  $\W_k$ is not reduced and $\ol{P}$ is a square mod $q$. Write $P=R^2+QH_0+\pi P_0$.
  Let $Q_1=Q+\pi^{\ell}R$. Then
  $$ F= Q_1^2+(\pi^sH+\pi^rH_0-2\pi^{\ell}R)Q_1+\pi^{r+1}P_1.$$
  This achieves Step 1.
  \medskip

  {\bf Step 2.} {\sl The polynomial $F$ admits a decomposition
    $$uF=Q^2+\pi^sQH+\pi^{2s}G$$
    with $\ol{H}$ prime to $q$ and $u\in 1+\pi\cO$. }

  Let us first show that $F$ admits a decomposition
  \begin{equation} \label{eq:qdh}
    F=Q^2+\pi^sQH+\pi^{2s}G.
  \end{equation}
  Suppose that this is not true. Starting with a relation~\eqref{eq:F-Q2} and using repeatedly Step 1,   we construct sequences $Q_n, P_n, H_n\in \cO[x_1,x_2,x_3]$ and
  $s_n\in {\mathbb N}^* \cup \{ +\infty \}$ such that
  $$F=Q_n^2+\pi^{s_n}Q_nH_n+\pi^{r+n}P_n$$
  and
  $$Q_{n+1}-Q_n\in \pi^{[(r+n)/2]}\cO[x_1,x_2,x_3],\ s_n\ge \min\{s_1, r,
  v(2)+r/2\}\text{ for all }n\ge 1\,.$$ The sequence $Q_n$ admits a limit
  $\hat{Q}\in \hat{\cO}[x_1, x_2,x_3]$. As $\hat{Q}$ reduces mod $\pi$ to
  $q\ne 0$, this implies that $Q_n+\pi^{s_n}H_n$ converges to some
  $\hat{R}\in\hat{\cO}[x_1,x_2, x_3]$ such that
  $F=\hat{Q}\hat{R}$. Contradiction because $C$ is geometrically integral.

  Now we prove the assertion of Step 2. Start with a
  relation~\eqref{eq:qdh}. If $\ol{H}$ is prime to $q$ then we are
  done. Otherwise, write $H=bQ+\pi H_0$ with $b\in \cO$. Then
  $$F=(1+b\pi^s)Q^2+\pi^{s+1}QH_0+\pi^{2s}G.$$
  Similarly to the above, applying repeatedly Step 1 to $(1+b\pi^{s})^{-1}F$,
  we find
  $$ (1+b\pi^{s})^{-1}F=Q_1^2+\pi^{s_1}Q_1H_1+\pi^{2s_1}G_1$$
  with $s_1\ge s+1$, $Q_1-Q\in \pi^{s}\cO[x_1,x_2,x_3]$. If $\ol{H}_1$ is
  prime to $q$, then we are done. Otherwise there exist $s_2\ge s_1+1$,
  $b_1\in \cO$ such that
  $$((1+b\pi^{s})(1+b_1\pi^{s_1}))^{-1}F=Q_2^2+\pi^{s_2}Q_2H_2+\pi^{2s_2}G_2.$$
  Again by the geometric integrality of $C$, we see that this process must
  stop after finitely many steps.  \medskip

  {\bf Step 3} Consider a relation given by Step 2. Let
  $$ \W=\Proj\left(\cO[x_1, x_2, x_3,y]/(Q-\pi^{s}y, y^2+Hy+G)\right).$$
  As $q$ is prime to $\ol{H}$, $\W_k$ is reduced by
  Lemma~\ref{lm:reduced2}(a), hence $\W\simeq \cC$ by
  Lemma~\ref{normality-proj}(b) and (c). This implies that $\ol{G}$ is prime
  to $q$ because $\cC_k$ is irreducible.
\end{proof}

\begin{remark}\label{rk:singular2}Suppose that $C$ has stable
  irreducible reduction over $K$. Then Theorems~\ref{th:goodmodel-f} and
  \ref{th:is2} hold after replacing smooth by stable, and in
  \ref{th:goodmodel-f}(c), $V_+(q)$ and $V_+(\ol{G})$ intersect with
  multiplicities at most $2$ instead of $1$.  The proof is exactly the same by
  using Remark~\ref{rk:singular}.  For the general case, there are several
  dozens of stable reduction types for curves of genus $3$. At the present
  moment it seems difficult to find characterizations for all types in a
  similar way as above. However, we think that some types can be dealed with
  the techniques we develop in this work. This will be the purpose of a future
  article.
\end{remark}

\begin{remark}
  Let $\cO$ be a PID with field of fractions $K$ and $C : F=0$ be a smooth
  projective plane quartic with good reduction everywhere over $K$. Gathering
  the local information above, one can give a smooth model of $C$ over $\cO$
  by
  $$\cC = \Proj \left(\cO[x_1, x_2, x_3, y]/(Q-\delta y, uy^2+Hy+G) \right) $$
  where $Q, H, G\in \cO[x_1, x_2, x_3]$ are of degrees $2, 2, 4$,
  $\delta \in \Oc$ is divisible by a prime $\p$ if and only if $C$ has
  hyperelliptic reduction at $\p$ and $u \in \Oc$ is congruent to $1$ modulo
  the primes for which $C$ has hyperelliptic reduction.
\end{remark}

\section{Geometric Invariant Theory and characterization of potentially good
  quartic reduction}

This section starts by gathering some definitions and results on the geometric
invariant theory over rings. The main reference is \cite{Se}. We derive
Corollary~\ref{cor:shabur} which unifies similar results present in the
literature under more restrictive hypotheses.  When $\Xb= \P^{14}$, the space
of plane quartics (represented by their $15$ coefficients), under the action
of $G=\SL_{3,\Z}$, this corollary says the following: given a discrete
valuation ring $\O$ with field of fractions $K$ and a smooth plane quartic
$\xs \in \Xb(K)$, after a possible finite extension of $K$, we can find an
integral model $\xfrak \in X(\O)$ such that its reduction is GIT-semi-stable
(whereas the reduction of a naive model of $\xs$ may be GIT-unstable).

We then need a second ingredient. Over an algebraically closed field, a
Homogeneous System Of Parameters (HSOP, see Definition~\ref{def:hsop}) is a
``minimal'' set of invariants which vanish exactly on the GIT-unstable
locus. We show in Section~\ref{sec:hsop} that for a class of nice rings $R$,
there exists such a set of invariants defined over $R$ which keep this
property over the algebraic closures of all residue fields of $R$
simultaneously. Remarks~\ref{rem:hsop1}, \ref{rem:hsop2} and \ref{rem:hsop}
will then help when working out such HSOP explicitly in
Section~\ref{sec:dixm-invar-char}.

Finally, in Section~\ref{sec:Shah}, we apply these results to the
characterization of the reduction type. Roughly speaking, the model $\xfrak$
is a minimal form (Definition~\ref{def:minimalform}) since we know that one of
the value of an invariant in the HSOP is not zero on the special fiber. Hence
its invariants are a minimal representative of the invariants of $\xs$
(Definition~\ref{def:minimalrepresentative}). Moreover, the valuations of this
representative can easily be computed from the values of the invariants
without knowing $\xfrak$.  Looking at the valuation of the discriminant with
respect to the HSOP (Definition~\ref{def:normalizedval}) will then give us the
reduction type of $\xfrak$ and therefore the type of the potential reduction
of $\xs$.

\subsection{Some results on Geometry Invariant Theory (GIT)}
\label{sec:appendix}

We let $S=\Spec R$ be a noetherian affine scheme, and $G$ be an affine smooth
group scheme over $S$ with connected and reductive geometric fibers (for
instance $\SL_{n,R}, \GL_{n,R}$).  Let $V$ be a finite rank free $R$-module
endowed with an action of $G$:
$$ G\to \mathrm{Aut}_R(V) $$
(morphism of group schemes). Then $G$ acts ``linearly'' on the projective
space $\P(V^{\vee})$, where $V^{\vee}$ is the dual of $V$.  Let $\Xb$ be a
$G$-stable closed subscheme of $\P(V^{\vee})$ of the form $\Xb=\Proj B$ where
$B$ is a $G$-stable homogeneous quotient of the symmetric algebra
$\mathrm{Sym}(V^{\vee})$ over $R$.  The ring of $G$-invariants $B^G$ is
$$B^G=\{ b\in B \mid \sigma(b\otimes 1_A)=b\otimes 1_A \}$$
where $A$ runs through the $R$-algebras and $\sigma\in G(A)$.  When $R$ is an
algebraically closed field $k$, the invariants of $B$ can be found using only
the $k$-rational points of $G$ and hence it corresponds to the usual
definition. Namely, denote by
$B^{G(k)}=\{ b\in B \mid g\cdot b=b, \ \forall g\in G(k)\}\supseteq B^G$, we
have the following result.

\begin{proposition} \label{punct-f} Let $k$ be an algebraically closed field
  and let $B$ be a $k$-algebra endowed with the action of a smooth affine
  algebraic group $G$. Then $B^{G}=B^{G(k)}$.
\end{proposition}

\begin{proof}
  Let $b\in B^{G(k)}$. Let $A$ be a $k$-algebra.  We have to show that
  $b\otimes 1_A\in B\otimes_k A$ is invariant by all $\sigma\in G(A)$. Let
  $h:=\sigma(b\otimes 1_A)-b\otimes 1_A \in B\otimes_k A$.

  First suppose that $A$ is reduced and of finite type over $k$.  For any
  closed point $y\in \Spec A$ the image of $h$ by the canonical homomorphism
  $$B\otimes_k A \to B\otimes_k k(y) = B\otimes_k k =B$$
  is $\sigma(y)(b\otimes 1_{k(y)})-b\otimes 1_{k(y)}=\sigma(y)(b)-b=0$, where
  $\sigma(y)\in G(k(y))=G(k)$ is the image of $\sigma$ by $G(A)\to G(k(y))$.
  Let $\{ e_i\}_i$ be a basis of $B$ as $k$-vector space. Then
  $h=\sum_i e_i\otimes a_i$ with $a_i\in A$ and we have $a_i(y)=0$ for all
  closed points $y\in \Spec A$. Hence $a_i=0$ and $h=0$.

  Now we consider the general case. Let $A$ be any $k$-algebra and let
  $\sigma\in G(A)$. Let us denote by $\sigma^{\#}: A_0:=\cO_G(G)\to A$ the
  corresponding $k$-algebra homomorphism. Let $\sigma_0\in G(A_0)$ be the
  identity morphism. Then the map $G(A_0)\to G(A)$ induced by $\sigma^{\#}$
  takes $\sigma_0$ to $\sigma$ and $h$ is the image of the element
  $\sigma_0(b\otimes 1_{A_0})-b\otimes 1_{A_0}$ by
  $\mathrm{Id}_B \otimes \sigma^{\#}$. Applying the previous case to the
  reduced $k$-algebra of finite type $A_0$, we get $h=0$.
\end{proof}

\begin{corollary} \label{punct-f2} Let $R$ be an integral domain
  with field of fractions $K$ and let $\bar{K}$ be an algebraic
  closure of $K$.
  Let $B$ be a flat $R$-algebra endowed with the action of a smooth
  affine algebraic group $G$ over $R$. Then
$$B^{G}=\{ b\in B \mid  b\otimes 1_{\bar{K}}
\in (B\otimes_R \bar{K})^{G(\bar{K})} \}.$$
\end{corollary}

\begin{proof} Let $b\in B$. To check that $b\in B^G$, as in the
proof of Proposition~\ref{punct-f}, it is enough to check that
$\sigma(b\otimes 1_{A_0})=b\otimes 1_{A_0}$ in $B\otimes_R A_0$ where
$A_0=\cO_G(G)$. As $B$ and $A_0$ are flat over $R$, the morphism
$B\otimes_R A_0\to (B\otimes_R \bar{K})\otimes_{\bar{K}} (A_0\otimes_R \bar{K})$ is
injective. So it is enough to check the desired equality over $\bar{K}$.
This then follows from Proposition~\ref{punct-f}.
\end{proof}

Let $R=k$ be an
algebraically closed field. A point $\xs  \in \Xb(k)$ is \emph{GIT-semi-stable}
if there is a rational point $\hat{\xs}$ in the affine cone $\hat{\Xb}$ of
$\Xb$ lying over $\xs$  such that $0\notin
\overline{G(k)\cdot\hat{\xs}}$
(Zariski closure in $\hat{\Xb}$). It is \emph{GIT-stable} if $G(k)\cdot \hat{\xs}$ is
closed of dimension $\dim G$ (\ie~the stabilizer of $G$ at $\hat{\xs}$ is
finite). When $\dim G>0$, GIT-stable points are also GIT-semi-stable.

For arbitrary $R$, a geometric point $\xs\in \Xb(k)$, where $k$ an
$R$-algebra which is an
algebraically closed field, is GIT-semi-stable
(resp. GIT-stable) if it is GIT-semi-stable (resp. GIT-stable) for the
action of $G\times_S \Spec k$ on $\Xb\times_S \Spec k$.
Then $\xs$ is
GIT-semi-stable if and only if there exists a homogeneous $f\in B^{G}$ of positive
degree such that $\xs\in D_+(f)$, see \cite[Thm. 1]{Se}.
As a consequence, the geometric points of the open subset
$\Xb^{\ds}:=\cup_f D_+(f)$ of $\Xb$,
where $f$ runs through the homogeneous elements of $B^G$ of positive
degrees, are exactly the GIT-semi-stable geometric points of $\Xb$.

\begin{theorem}[\cite{Se}, Theorem 4, p. 269 and Remarks 8, 9, p. 269,
271]\label{theorem4} Let $R$ be a noetherian ring. Then
we have the following properties.
\begin{enumerate}
\item \label{affine} The
inclusion $B^G\subseteq B$ induces an affine surjective morphism
$$\phi :  \Xb^{\ds}\to \Yb:=\Proj (B^{G}).$$
This morphism is a \emph{categorical quotient}: any $G$-invariant
morphism $\Xb^{\ds}\to \Zb$ with trivial action on $\Zb$ factors
uniquely through $\Yb\to \Zb$. The formation of this quotient
commutes with flat base changes $R\to R'$.
\item \label{surj} If $\Zb$ is a $G$-stable closed subscheme of $\Xb$, then $\phi(\Zb)$
is closed in $\Yb$.
\item\label{quot} For any pair of geometric points $\xs_1, \xs_2\in \Xb^{\ds}(k)$, we have
$\phi(\xs_1)=\phi(\xs_2)$ if and only if
$$\overline{G(k)\cdot \xs_1}\cap\overline{G(k)\cdot \xs_2}\ne\emptyset.$$
\item There is an open subset $\Yb^s\subseteq \Yb$ such that the geometric
points of $\Xb^{\mathrm s}:=\phi^{-1}(\Yb^s)$ are exactly
the GIT-stable points. In particular, if $\xs_1, \xs_2\in \Xb^s(k)$ are two geometric
points, then $\phi(\xs_1)=\phi(\xs_2)$ if and only if $\xs_2\in G(k)\cdot
\xs_1$.
\item \label{th;proj} If $R$ is an excellent ring\footnote{See
  [EGA], IV.7.8.2. Localizations of algebras of finite type over a
  field, Dedekind domains of characteristic $0$, complete noetherian
  local rings are excellent ([EGA], IV.7.8.3(ii)-(iii)). Excellent
rings are universally Japanese ([EGA], IV.7.8.3(vi)).}, then
$B^G$ is of finite type over $R$ and $\Yb$ is projective over $R$.
\end{enumerate}
\end{theorem}

\begin{remark}
In \cite[Appendice]{Bur}, some clarifications are added to the proof of \cite[Prop.8]{Se}.
\end{remark}

The categorical quotient $\Yb$ is denoted by $\Xb^{\ds}\dq G$.
Let $\ys\in (\Xb^{\ds}\dq G)(k)$ be a geometric point.
Then $\phi^{-1}(\ys)$ is a union
of closures of orbits. Any orbit of smallest dimension contained
in $\phi^{-1}(\ys)$ is closed, hence contained in (and equal to) the
intersection of all these closures. It is a \emph{minimal orbit}.

The following result appears in \cite[Prop.2.1]{shah} and
    \cite[Lem.5.3]{mum-ens} in equal characteristics setting, in
    \cite[p.122]{Bur} over the $p$-adics for $\Xb=\P(V)$ and in
    \cite{silverman-dyna} and \cite{szpiro-semi} in a dynamical
    context.

\begin{corollary} \label{cor:shabur}
Let $R=\O$ be a discrete valuation ring with field of fractions $K$ and such that
the categorical quotient $\Xb^{\ds}\dq G$ is of finite type (hence
projective) over $\O$.
Let $\xs\in \Xb^s(K)$. Then there exists an extension $\O\subseteq \O'$
of discrete valuation rings with $K'=\Frac(\O')$ finite over $K$,  and an integral point $\xfrak \in \Xb^{\ds}(\O')$ such that
$\xfrak_{K'}\in G(K')\cdot \xs$. Moreover,
we can ask the image of the closed point of $\xfrak$ to belong to a minimal
orbit.
\end{corollary}
\begin{proof}
  Let $\ys=\phi(\xs)\in \Yb^s(K)$. Let
  $\Gamma=\overline{\{ \ys \}}\subseteq \Yb$.  As $\Yb/S$ is projective,
  $\Gamma\to S$ is an isomorphism.  Let $W_K$ be the closed subset
  $\phi^{-1}(\Gamma)_K=\phi^{-1}(\ys)$ of $\Xb_K^{\ds}$ endowed with the
  reduced subscheme structure. Let
  $W=\overline{\phi^{-1}(\Gamma)_K}\subseteq \Xb^{\ds}$ be its
  scheme-theoretical closure.  As $W_K$ is a homogeneous space under $G_K$,
  hence irreducible, then so is $W$.  Therefore $W$ is an integral scheme
  dominating $S=\Spec \O$, hence flat over $S$.  It is a $G$-stable closed
  subscheme of $\Xb^{\ds}$.  By Theorem \ref{theorem4}(\ref{surj}),
  $\phi(W)\subseteq \Gamma$ is closed hence equal to $\Gamma$. So $W\to S$ is
  surjective.  By \cite{liu-book}, Prop. 10.1.36, any closed point $w$ of the
  (non-empty) special fiber of $W\to S$ belongs to an irreducible closed
  subscheme $\Gamma'\subseteq W$ quasi-finite and surjective over $S$.

Let $\Spec \O'$ be the localization at a closed point
of the normalization of $\Gamma'$. Note that $\O'$ is a
discrete valuation ring.
The canonical morphism $\xfrak: \Spec \O'\to \Xb^{\ds}$ is an integral point as
we are looking for: its generic point belongs to the same geometric
orbit as $\xs$, so $\xfrak_{K'}\in G(K').\xs$ after a finite extension of $K'$
if necessary. Moreover,
the closure of the orbit of $w$ is contained in the closed fiber
$W_s$,  so the minimal orbit is also contained in $W_s$. Therefore we
can chose $w$ in the minimal orbit.
\end{proof}

Now let us say a few words on the fibers of $\Xb^{\ds}\dq G\to S$.
Let $G, \Xb$ and $S=\Spec R$ be as in Theorem~\ref{theorem4}.
For any $T\to S$ with $T$ affine noetherian,
we have $(\Xb\times_S T)^{\ds}=\Xb^{\ds}\times_S T$.
By the categorical quotient property, we have a canonical morphism
\begin{equation}  \label{eq:comp}
(\Xb^{\ds}\times_S T)\dq G_T \to (\Xb^{\ds}\dq G)\times_S T.
\end{equation}
This is an isomorphism if $T\to S$  is flat
\cite[p. 271]{Se}.
In general, this morphism induces a bijection on the geometric points
by Theorem~\ref{theorem4}(\ref{quot}). If both sides are projective
schemes over $T$ (e.g. $S, T$ are both excellent), then the morphism
is {projective \cite[Corollary 3.3.32(e)]{liu-book}, hence finite
\cite[Corollary 4.4.7]{liu-book} and bijective.
In particular, if $R$ is excellent, then for any geometric point
$s\in S(k)$ we get a finite
bijective, and hence homeomorphic, morphism of projective schemes over $k$:
$$(\Xb_s)^{\ds}\dq G_s\to (\Xb^{\ds}\dq G)_s.$$

\begin{remark} \label{quotient-ft} Suppose that $\Xb$, $G$ and its action on
  $\Xb$ are all defined over $\Z$. Then for any integral domain $A$, the
  categorical quotient $(\Xb_A)^{\ds}\dq G_A$ is projective over $A$. Indeed,
  if $A$ has characteristic zero, then it is flat over $\Z$ and
  $(\Xb_A)^{\ds}\dq G_A$ is obtained by base change from $\Xb^{\ds}\dq G$
  which is projective thanks to
  Theorem~\ref{theorem4}\eqref{th;proj}. Otherwise $A$ contains a prime field
  $\F_p$, and $(\Xb_A)^{\ds}\dq G_A$ is obtained by the (flat) base change
  from $(\Xb_{\F_p})^{\ds}\dq G_{\F_p}$ to $A$.
\end{remark}

\subsection{Homogeneous system of parameters} \label{sec:hsop}

\begin{definition} \label{def:hsop}
Let $E$ be a graded algebra over a ring $R$.
\begin{enumerate}
\item A family of homogeneous elements $f_0, f_1, \dots, f_m\in E_+$ is called
  a \emph{Homogeneous System Of Radical Generators (HSORG) over $R$} if
  $V_+(f_0, \dots, f_m)=\emptyset$ in $\Proj E$. This is equivalent to the
  inclusion ${E_+}\subseteq \sqrt{(f_0, \dots, f_m)}$, where $E_+$ is the
  irrelevant ideal of $E$.
\item If the fibers of $\Proj E\to \Spec R$ all have the same dimension
  $d\ge 0$, a \emph{Homogeneous System Of Parameters (HSOP) over $R$} of $E$
  consists in a HSORG with $d+1$ elements. Note that $d+1$ is the smallest
  possible cardinality of a HSORG.
\end{enumerate}
\end{definition}

\begin{remark} \label{rem:hsop1} Let $A$ be an $R$-algebra. Any HSORG
  $\{f_0, \dots, f_m\}$ of $E$ gives rise to a HSORG
  $\{f_0\otimes 1, \dots, f_m\otimes 1 \}$ of $E\otimes_R A$. The converse is
  true if $R\to A$ is faithfully flat (e.g. a field extension).  Under the
  condition (ii) of the definition, the fibers of the morphism
  $\Proj (E\otimes_R A)\to \Proj A$ have dimension $d$ and
  $\{ f_0\otimes 1, \dots, f_d\otimes 1\}$ is a HSOP of $E\otimes_R A$.
  Conversely, for $f_0, \dots, f_m$ to be a HSORG (resp. HSOP) of $E$, it is
  enough that their images in $E \otimes_R k(t)$ form a HSORG (resp. HSOP) for
  all points $t\in \Spec R$. If $R=\O$ is a discrete valuation ring, by the
  properness of $\Proj E\to \Spec R$, a list of homogeneous elements of $E$ is
  a HSORG (resp. HSOP) over $\O$ if and only if their images in
  $E \otimes_{\O} k$ form a HSORG (resp. HSOP) over the algebraic closure of
  the residue field $k$.
\end{remark}

\begin{remark} \label{rem:hsop2} Let $B$ be a homogeneous algebra over a field
  $k$, let $G$ be a reductive algebraic group over $k$ acting on $B$. Let
  $f_0, \dots, f_d\in B^G=E$ be homogeneous with $d=\dim \Proj E$. Then they
  form a HSOP over $k$ of $E$ if in $\Spec B$, we have
  $V(f_0, \dots, f_d)=V(\m B)$ (the \emph{null-cone}) where $\m=E_+$. Moreover
  $E$ will be a finite algebra over $k[f_0, \dots, f_d]$, see \cite[Lemma
  2.4.5.]{DerKem}.
\end{remark}

\begin{remark} \label{rem:hsop} Let $R, B, G$ be as in the beginning of this
  section. Let $A$ be a noetherian $R$-algebra.
  \begin{enumerate}
  \item Any HSORG of $B^G$ generates a HSORG of $(B\otimes_R A)^{G_A}$ by the
    surjectivity of the morphism \eqref{eq:comp}.
  \item If moreover $B^G$ and $(B\otimes_R A)^{G_A}$ are of finite type
    respectively over $R$ and $A$, then the same statement holds for any HSOP
    of $B^G$ because for any $t\in \Spec A$ lying over some $s\in S$,
    $(\Proj (B\otimes_R A)^{G_A})_{k(t)}$ is homeomorphic to
    $(\Proj B^G)_s \times_{\Spec k(s)} \Spec k(t)$, hence they have the same
    dimension.
  \item Let $R=\O$ be a discrete valuation ring with residue field $k$, and
    suppose $B^G$ is of finite type over $\O$.  Using
    Proposition~\ref{punct-f} and the above remarks, we see that a list of
    homogeneous elements of $B^G$ is a HSOP of $B^G$ if and only if their
    images in $B \otimes_{\O} \bar{k}$ define the zero set
    $V((B \otimes_{\O} \bar{k})^{G(\bar{k})}_+ (B \otimes_{\O} \bar{k}))$ or
    equivalently they form a HSOP of $(B\otimes_{\cO} \bar{k})^{G(\bar{k})}$
    over $\bar{k}$. We can then use classical methods of GIT over an
    algebraically closed field to prove that a given list of homogeneous
    elements of $B^G$ is a HSOP of $B^G$.
  \end{enumerate}
\end{remark}

A natural question is when a HSOP exists.  A \emph{pictorsion ring}
(\cite{GLL}, Definition 0.3) is a commutative ring $R$ such that for any
finite homomorphism $R\to A$, the group $\mathrm{Pic}(A)$ is of torsion. The
ring of integers of a finite extension of $\mathbb Q$, the ring of regular
functions of an affine connected regular curve over a finite field, and
semi-local rings are pictorsion.

\begin{proposition} \label{hsop} Let $R$ be a noetherian pictorsion ring. Let
  $E$ be a homogeneous $R$-algebra such that the fibers of
  $\Proj E\to \Spec R$ all have the same dimension $d\ge 0$. Then $E$ has a
  HSOP over $R$.
\end{proposition}

\begin{proof} This is done in the proof of Theorem 8.1 in \cite{GLL}.
\end{proof}

Now we return to the early situation of Section~\ref{sec:appendix} where
$\Xb=\Proj B$ has linear $G$-action.

\begin{corollary} \label{cor:hsop} Suppose that $R$ is excellent, pictorsion,
  and $S=\Spec R$ is connected. Assume moreover that $\Xb$ is flat over $S$,
  that $\Xb_s$ is irreducible and has a GIT-stable point for all $s\in S$.
  Then $B^G$ admits a HSOP over $R$.
\end{corollary}

\begin{proof} All fibers of $\Xb/S$ (resp. of $G/S$) have the same dimension
  $n$ (resp. $m$). For all $s\in S$, as $\Xb_s$ is irreducible, its non-empty
  open subsets $(\Xb_s)^{\ds}$ and $(\Xb_s)^{\mathrm{s}}$ are irreducible of
  dimension $n$. Therefore $(\Xb_s)^{\mathrm{s}}\dq G_s$ has dimension
  $\dim (\Xb_s)^{\mathrm{s}} -\dim G_s=n-m$. Let
  $\Yb=\Xb^{\ds}\dq G=\Proj (B^G)$. For all $s\in S$, the fiber $\Yb_s$, which
  is homeomorphic to $(\Xb_s)^{\ds}\dq G_s$, is irreducible and contains an
  open subset of dimension $n-m$.  So $\dim \Yb_s=n-m$ for all $s\in S$. Now
  we can apply Proposition~\ref{hsop}.
\end{proof}

\subsection{Characterization of potentially good quartic reduction} \label{sec:Shah}

Let $R=\O$ be a discrete valuation ring with field of fractions $K$ and
$\Xb=\Proj B$ such that $B^G$ is of finite type over $\O$.  This last property
implies that $B^G$ admits a HSORG $\Iv$.  By definition, a geometric point
$\xs \in \Xb(k)$ is GIT-unstable (i.e. not GIT-semi-stable) if and only if
$\xs \in V_+(I)_k$ in $\Xb_k$. Then Corollary~\ref{cor:shabur} has the
following translation.

\begin{proposition} \label{prop:ssred} With the notation of
  Corollary~\ref{cor:shabur}, after a possible finite extension of $K$, there
  exists $\xfrak \in \Xb^{\ds}(\O)$ such that $\xfrak_K \in G(K) \cdot \xs$
  and $\Iv(\xfrak)$ is a minimal representative of $\Iv(\xs)$.
\end{proposition}

\begin{proof} After possibly a finite extension of $K$, let
  $\xfrak \in \Xb^{\ds}(\O)$ be as given by Corollary~\ref{cor:shabur}. As
  $V_+(I)=\emptyset$ in $\Proj B^G$, we have $\Xb^{\ds}=\cup_{f\in I} D_+(f)$,
  hence $\xfrak\in D_+(f)$ for some $f\in I$. Therefore $v(f(\xfrak))\in \O^*$
  and $\Iv(\xfrak)$ is a minimal representative of $\Iv(\xs)$.
\end{proof}

Let $d>2,\, n \geq 1$ be integers and $N={n+d \choose d}$. Let
$\Xb \simeq \P^{N}=\Proj \Z[\xv]$, the space of hypersurfaces of degree $d$ in
$\P^n$ under the classical action of $G=\SL_{{n+1},\Z}$. As $\Z$ is a PID, any
$X \in \Xb(\Z)$ is represented by an $(n\!+\!1)$-ary form $F$, unique up to
multiplication by $\pm 1$.

Note that $\Z$ is excellent and pictorsion and that the scheme $\Xb$ satisfies
the hypotheses of Corollary~\ref{cor:hsop} hence $\Z[\xv]^G$ admits a HSOP
over $\Z$. Moreover, there exists an invariant $D_{\ell} \in \Z[\xv]^G$ of
degree $\ell=(n+1)(d-1)^n$, called the \emph{discriminant}, such that any
hypersurface $F=0$ over a ring $R$ is smooth if and only if $D_{\ell}(F)$ is
invertible in $R$ \cite[p.426]{gelfand}, \cite{demazure}.  Over any discrete
valuation ring $\O$, $\O[\xv]^{G_\O}$ is finitely generated as $\O$-algebra by
Remark~\ref{quotient-ft}, and any HSOP of $\Z[\xv]^G$ then gives rise to a
HSOP of $\O[\xv]^{G_\O}$ according to remark~\ref{rem:hsop}.

\begin{corollary} \label{cor:disc0} Let $K$ be a discrete valuation field with
  valuation $v$, valuation ring $\O$ and residue field $k$ of characteristic
  $p \geq 0$.  Let $\Iv$ be a HSORG for the graded algebra
  $\Z[\xv]^{\SL_{n+1,\Z}}$ for hypersurfaces of degree $d$ in $\P^n$.  Let
  $X/K : F=0$ be a smooth hypersurface of degree $d$ in $\P^n$ and $D_{\ell}$
  the discriminant as defined above. Then $v_{I}(D_{\ell}(F))=0$ if and only
  if, after a possible finite extension of $K$, there exists a hypersurface
  $\mathcal{X}/\O$ such that $\mathcal{X}_K \simeq X$ and $\mathcal{X}_k$ is a
  smooth hypersurface of degree $d$.
\end{corollary}
\begin{proof}
  Let $X/K : F=0$ be a smooth hypersurface of $\P^N$ of degree $d$. By
  \cite[Prop.4.2]{mumford-fogarty}, $X/K$ is GIT-stable. After a possible
  finite extension of $K$, by Proposition~\ref{prop:ssred}, there exists a
  hypersurface $\mathcal{X}/\O : F_0=0$ in $\P^N$ of degree $d$, isomorphic to
  $X$ over $K$ and such that $\underline{I}(F_0)$ is a minimal representative
  over $\O$. Therefore if $v_I(D_{\ell}(F))=0$, then $v(D_{\ell}(F_0))=0$ and
  $\mathcal{X}_k$ is a smooth hypersurface of degree $d$. This proves the
  direct implication. For the converse implication, if $\mathcal{X}/\O : F_0=0$
  is smooth, necessarily in $\P^n$ of degree $d$, then
  $D_{\ell}(\bar{F}_0) \ne 0$ i.e. $v(D_{\ell}(F_0))=0$ which implies, since
  $\mathcal{X}_K \simeq X$, that $v_I(D_{\ell}(F))=0$ as well.
\end{proof}

As a particular case, we get our Theorem~\ref{th:GNHR}.  Indeed, notice that
by the discussions in \S \ref{general-facts}, if a smooth quartic curve $X/K$
has potentially good quartic reduction, then after possibly a finite extension
of $K$, $X$ is isomorphic to the generic fiber of some smooth quartic curve
over $\O$.

\begin{theorem}[See Theorem~\ref{th:GNHR}] \label{th:gnhr} Let $K$ be a
  discrete valuation field with valuation $v$, valuation ring $\O$ and residue
  field $k$ of characteristic $p \geq 0$. Let $\Iv$ be a list of invariants
  which is a HSORG for the ternary quartic forms under the action of
  $\SL_{3,\O}$.  Then a smooth plane quartic $C/K : F=0$ has potentially good
  quartic reduction if and only if $v_{I}(D_{27}(F))=0$.
\end{theorem}

\begin{remark}
  It is unclear how to effectively compute the potential smooth quartic model.
  In some examples, the extension of $K$ to be considered can be huge. For
  instance, $y^3=x^4-x$ has good quartic reduction over an extension of
  minimal degree $108$ of $\Q_3$ (Bouw-Wewers, private communication).
\end{remark}

The previous result will be effective once we have determine a HSOP. This is
the main result of Section~\ref{sec:dixm-invar-char}. Meanwhile, let us work
out a simpler case, which will also be useful for the characterization of
potentially good hyperelliptic reduction. When $p \ne 2, 3, 5$ and 7,
\cite[Prop.1.9]{LR11} shows that the set of Shioda invariants
$j_2,\ldots,j_{10}$ is a HSORG over $\O$ for the graded algebra of invariants
of binary octics under the action of $\SL_{2,\O}$. Actually, taking into
account that the vanishing of $j_2,\ldots,j_7$ implies the vanishing of the
discriminant $D_{14}$, the proof of loc. cit. shows that
$\operatorname{\underline{Sh}}=(j_2,\ldots,j_7)$ is a HSOP over $\O$ when
$p=0$ or $p>7$.

Corollary~\ref{cor:disc0} specializes in the following.

\begin{corollary}\label{cor:Shiodas}
  Let $K$ be a discrete valuation field with valuation $v$, valuation ring
  $\O$ and residue field $k$ of characteristic $p \ne 2, 3, 5$ and $7$. Let
  $C:\,y^2=f(x)$ over $K$ be a hyperelliptic curve of genus $3$ with
  $f\in\mathcal{O}[x]$. Then $C$ has potentially good reduction if and only if
  $v_{\operatorname{Sh}}(D_{14}(f))=0$, with
  $\operatorname{\underline{Sh}}=(j_2,...,j_{7})$.
\end{corollary}

\begin{remark}
  For binary octics, Basson \cite{basson} exhibits a HSOP of degree
  $(3,\, 4,\, 5,\, 6,\, 10,\, 14)$ when $p=7$ and one of degree
  $(4,\, 5,\, 6,\, 7,\, 8,\, 9)$ when $p=3$.  Similar techniques for $p=5$
  lead to a HSOP over $\O$ of degree $(4,\, 6,\, 6,\, 12,\, 14,\, 20)$. We can
  therefore extend Corollary~\ref{cor:Shiodas} to all characteristics
  different from $2$ using these sets instead.
\end{remark}

\section{Homogeneous system of parameters over $\O$}
\label{sec:homog-syst-param}

\subsection{Computation of the various HSOP}
\label{sec:dixm-invar-char}
Let $\O$ be an DVR with residue field $k$ of characteristic $p$.  In order to
make Theorem~\ref{th:GNHR} effective, we need an explicit HSOP over $\O$ for
ternary quartic forms under the action of $\SL_{3,\O}$.  We have seen in
Section~\ref{sec:Shah} that such a HSOP exists and is even defined over $\Z$.
Dixmier proves~\cite{dixmier} that the $7$ invariants $I_{3}$, $I_{6}$,
$I_{9}$, $I_{12}$, $I_{15}$, $I_{18}$ and $I_{27}$ form a HSOP over a field of
characteristic $0$. The result does not carry in characteristic $p>0$ for all
$p$ but the strategy he followed is valid in positive characteristic, although
it requires to be handled with care.

We work out the computations with the computer algebra software \Magma{}. The
invariants are given by polynomials over $\Q$ in the Dixmier-Ohno
invariants. They are found by an evaluation/interpolation strategy over
$\Z/p^n \Z$ for increasing powers $n$ (the largest computation being in degree
$81$ over $3^{19}$). We will not give details about this search here. Once
given, it is indeed ``enough'' to check that that the invariants we obtained
satisfy two conditions: 1) as polynomials in the $15$ coefficients of a
ternary quartic forms, they have \emph{integral} coefficients; 2) once reduced
modulo $p$, these polynomials form a HSOP over $\bar{\F}_p$. We could prove
these two facts for our invariants in all characteristics $p \ne 3$. Let us
say quickly some words about 1). For a given characteristic, one can transform
a general plane quartic into a quartic with only $7$ parameters by an integral
change of variables (when $p>3$, one can take Shioda's models in
\cite{shioda-quartic}). When the degree of the invariant is less or equal to
$27$, it is then possible to write down the invariant extensively in these 7
variables and check directly the integrality of each coefficient. For $p=3$,
we cannot prove integrality for the degree $54$ and $81$ invariants and we
therefore only present the result as Conjecture~\ref{conj:hsop3}. Dealing with
2) is somehow classical as we will see in the proof of the following result.

\begin{theorem}\label{th:domodp}
  Let $\O$ be a DVR with residue field of characteristic $p \neq 3$.  A HSOP
  over $\O$ for ternary quartic forms under the action of $\SL_{3,\O}$ is
  formed by
  \begin{itemize}
  \item
    $(I_3^{(2)},I_9^{(2)},I_{12}^{(2)},I_{15}^{(2)},I_{18}^{(2)},I_{24}^{(2)},D_{27})$
    when $p=2$. The large expressions can be found in
    \cite[\texttt{DixmierHSOP.txt}]{HSOPList};
  \item    \begin{math}
      ( I_{3},\  I_{6},\  I_{9}^{(5)},\  I_{12},\  J_{15}^{(5)},\  I_{18},\  I_{27} )
    \end{math}
    where $I_{9}^{(5)} = (J_9 + 3\,I_9)/5$ and
    \begin{small}
      \begin{multline*}
        J_{15}^{(5)} =  (-64\,{{I_{3}}}^{3}{I_{6}}-24\,{I_{3}}\,{{I_{6}}}^{2}+39\,{I_{3}}\,{I_{12}}-11\,{I_{3}}\,{J_{12}}+42\,{I_{6}}\,{I_{9}}-21\,{I_{6}}\,{J_{9}}-1143\,{I_{15}}+{J_{15}})/5^3\\
        +({{253}}\,{{I_{3}}}^{2}{I_{9}}-{{79}}\,{{I_{3}}}^{2}{J_{9}})/5^4\,\,\text{when }p=5;
      \end{multline*}
    \end{small}
  \item   \begin{math}
      ( I_{3},\, I_{6},\, I_{9} - J_{9},\, I_{12},\, I_{15},\, I_{18},\, I_{27} )
    \end{math}
    when $p = 7$, $19$, $47$, $277$ and $523$;
  \item  \begin{math}
      ( I_{3},\, I_{6},\, I_{9},\, I_{12},\, I_{15},\, I_{18},\, I_{27} )
    \end{math}
    otherwise.
  \end{itemize}
In particular for $p \geq 7$, the Dixmier-Ohno invariants form a
HSORG.
\end{theorem}
\begin{proof} %
  By Remark~\ref{rem:hsop}, it is enough to prove that the previous sets are
  HSOP over the algebraic closure of the residue field of $\O$.  To do this,
  we first find all ternary quartic forms $F$ in characteristic $p$ such that
  $I(F) = 0$ for all $I \in \Iv$, where $\Iv$ is the set of invariants given
  in the theorem.  We write $F$ as
  \begin{multline*}
         a_{04}\,{x_2}^{4}+a_{13}\,x_1{x_2}^{3}+a_{22}\,{x_1}^{2}{x_2}^{2}+a_{31}\,{x_1}^{3}x_2+a_{40}\,{x_1}^{4}\, +
      ( a_{03}\,{x_2}^{3}+a_{12}\,x_1{x_2}^{2}+a_{21}\,{x_1}^{2}x_2+a_{30}\,{x_1}^{3} )\,x_3\, \\+
          ( a_{02}\,{x_2}^{2}+a_{11}\,x_1x_2+a_{20}\,{x_1}^{2} )\,{x_3}^{2}\,+
      ( a_{01}\,x_2+a_{10}\,x_1 )\,{x_3}^{3}\,+  a_{00}\,{x_3}^{4}\,.
  \end{multline*}
  Then, the set $\Iv$ is a HSOP in characteristic $p$ if and only if any such
  $F$ is GIT-unstable, or equivalently $F$ has a triple point or consists of a
  cubic and an inflectional tangent line~\cite[§1.12]{mum-ens}. Let us start
  with $p \ne 2$ (and $3$).
  Since $I_{27}(F) =2^{40} D_{27}(F)= 0$, the quartic $F=0$ has a singular
  point, that we can, up to a linear change of variables, set to
  $\Pi=(0:0:1)$. Following step by step~\cite{dixmier}, we can use the
  remaining degrees of freedom to restrict to the six different cases that we
  sum up in Table~\ref{tab:singmodels} and that we will consider one after the
  other. We checked that the linear transformations behind these
  normalizations are valid over fields of characteristic $p \geqslant 5$ and
  $0$ (see~\cite[\texttt{HSOPDetailedProof.txt}]{HSOPList}).
\begin{table}[htbp]
  \centering
  \setlength{\tabcolsep}{2pt}
  \renewcommand{\arraystretch}{1.0}
  \begin{tabular}{c|l|l|l|}
    Case & \multicolumn{2}{l|}{\ Condition}& \ $F(x_1,x_2,x_3)$\\\hline\hline
    0& \multicolumn{2}{l|}{$\Pi$ of order $\geqslant 3$} &
    $
    \begin{array}{c}
      a_{04}\,{x_2}^{4}+a_{13}\,x_1{x_2}^{3}+a_{22}\,{x_1}^{2}{x_2}^{2}+a_{31}\,{x_1}^{3}x_2+a_{40}\,{x_1}^{4}+\\
      ( a_{03}\,{x_2}^{3}+a_{12}\,x_1{x_2}^{2}+a_{21}\,{x_1}^{2}x_2+a_{30}\,{x_1}^{3} )\,x_3
    \end{array}
    $\\\hline\hline
    1.1& \multirow{2}%
    {0.2\textwidth}%
    {$\Pi$ of order 2 \& $x_1=0$ is the only tangent at $\Pi$}%
    & $\begin{array}{c}\ a_{03} = 0\end{array}$ &
    $\begin{array}{c}
      a_{04}\,{x_2}^{4}+a_{13}\,{x_2}^{3}x_1+a_{22}\,{x_2}^{2}{x_1}^{2}+a_{31}\,x_2{x_1}^{3}+a_{40}\,{x_1}^{4}+\\
      ( a_{12}\,{x_2}^{2}x_1+a_{21}\,x_2{x_1}^{2}+a_{30}\,{x_1}^{3} )\, x_3+{x_1}^{2}\,{x_3}^{2}
    \end{array}$%
    \\\cline{1-1}\cline{3-4}
    1.2 & & $\begin{array}{c}\ a_{03} \neq 0\end{array}$ &
    $\begin{array}{c}
      a_{04}\,{x_2}^{4}+a_{13}\,{x_2}^{3}x_1+a_{22}\,{x_2}^{2}{x_1}^{2}+a_{31}\,x_2{x_1}^{3}+a_{40}\,{x_1}^{4}+\\
      {x_2}^{3}x_3+{x_1}^{2}{x_3}^{2}
    \end{array}$
    \\\hline\hline
    2.1& \multirow{5}%
    {0.2\textwidth}%
    {$\Pi$ of order 2 \& $x_1=0$, $x_2=0$ are both tangent at $\Pi$}%
    &
    $\begin{array}{c}
      a_{03} \neq 0\\
      \&\,a_{30} \neq 0
    \end{array}$
    &
    $\begin{array}{c}
      a_{04}\,{x_2}^{4}+a_{13}\,{x_2}^{3}x_1+a_{22}\,{x_2}^{2}{x_1}^{2}+a_{31}\,x_2{x_1}^{3}+a_{40}\,{x_1}^{4}+\\
      ( {x_2}^{3}+{x_1}^{3} ) x_3
      +x_2x_1{x_3}^{2}
    \end{array}$
    \\\cline{1-1}\cline{3-4}
    2.2 & &
    $\begin{array}{c}
      a_{03} = 0\\
      \&\,a_{30} \neq 0
    \end{array}$ &
    $\begin{array}{c}
      a_{04}\,{x_2}^{4}+a_{13}\,{x_2}^{3}x_1+a_{22}\,{x_2}^{2}{x_1}^{2}+a_{31}\,x_2{x_1}^{3}+a_{40}\,{x_1}^{4}+\\
      {x_1}^{3}x_3+x_2x_1{x_3}^{2}
    \end{array}$
    \\\cline{1-1}\cline{3-4}
    2.2 & &
    $\begin{array}{c}
      a_{03} = 0\\
      \&\,a_{30} = 0
    \end{array}$ &
    $\begin{array}{c}
      a_{04}\,{x_2}^{4}+a_{13}\,{x_2}^{3}x_1+a_{22}\,{x_2}^{2}{x_1}^{2}+a_{31}\,x_2{x_1}^{3}+a_{40}\,{x_1}^{4}+\\
      x_2x_1{x_3}^{2}
    \end{array}$\\\hline\hline
  \end{tabular}
  \smallskip
  \caption{Singular quartics in characteristic $p \geqslant 5$ and $0$}
  \label{tab:singmodels}
\end{table}

\noindent$\bullet$ Case $0$ yields quartics with $\Pi$ being a triple point
and these curves are GIT-unstable.  \smallskip

\noindent$\bullet$ Case $1.1$ is also easily manageable. We use the computer
algebra software \Magma{} to shortcut some cumbersome computations (that
Dixmier made by hand over the rationals !). In particular, we compute a
Gr\"obner basis \emph{over $\Z$} (based on the computation of normal Hermite
forms of integral Macaulay matrices) of the algebraic system
$\{\,I(F) = 0 \,\}_{I\in\Iv}$ for the following lexicographic order:
$$a_{30}>a_{22}>a_{21}>a_{31}>a_{13}>a_{40}>a_{04}.$$ We get
\begin{math}
  \{{a_{12}}^{8}+144\,{a_{04}}^{4},\ %
  {a_{12}}^{4}a_{04}+144\,{a_{04}}^{3},\ %
  2\,{a_{12}}^{2}+12\,a_{04},\ %
  360\,{a_{04}}^{3},\ %
  720\,{a_{04}}^{2}\}\,.
\end{math}
In characteristic $p>5$ and $p=0$, we thus have $a_{04} = a_{12} = 0$\,. This
yields quartics $F=0$ which are the union of a cubic curve and an inflectional
tangent at $\Pi$, which we can assume to be $x_1=0$, thus GIT-unstable
quartics.  The exceptional prime $p=5$ will be handled later on its
own.\smallskip

\noindent$\bullet$ Case $1.2$ is done in a similar manner. Gr\"obner basis can
also be computed over $\Z$.  When $p>5$ and $p=0$, we derive from it that
$a_{04}=a_{31}=a_{22}=0$, and
\begin{math}
  \{{a_{13}}^{2}+9953670985\,a_{40},\ 2^{4}\cdot 3^{13}\cdot 5^{2} \cdot 7\cdot 19\cdot 47 \,a_{40}\}\,.
\end{math}
Thus, for $p\neq 2$, $3$, $5$, $7$, $19$ and $47$, we have
$a_{13}=a_{40}=0$. This yields quartics of the same type as in Case $1.1$, and
thus, GIT-unstable. These exceptional primes will be handled later on their
own.\smallskip

\noindent$\bullet$ Case $2.1$ is the most difficult case, at least
computationally speaking. From $I_3(F) = I_{6}(F) = 0$, we easily deduce that
$a_{22}=-9/2$ and $a_{04}a_{40}=9/4$. Now, the equations $I(F)=0$ for the
other invariants $I$ yield more algebraic relations between the remaining
variables, but of higher degrees, and we were not able anymore to compute an
integral Gr\"obner basis. Instead, we find sufficient conditions on the
remaining $a_{ij}$ by solving with resultants. We eliminate, first $a_{04}$,
then $a_{40}$ and finally $a_{31}$ (resp.  $a_{13}$, $a_{31}$ and $a_{04}$),
and obtain a polynomial of degree 48 in $a_{13}$ (resp. of degree 28 in
$a_{40}$),
\begin{displaymath}
  N_{13}\, (4\,a_{13} + 15)^{16}\, (16\,a_{13}^2 - 60\,a_{13} + 225)^{16}\ \ \ \ %
  (\text{resp. }N_{40}\, a_{40}^4\, (2\,a_{40} + 3)^8\, (4a_{40}^2 - 6\,a_{40} + 9)^8\,)\,,
\end{displaymath}
where $N_{13}$ (resp. $N_{40}$) is an integer, too large to be completely
factorized. We can have $N_{13}$ and $N_{40}$ both equal to zero modulo $p$,
but this can only happen for the few prime divisors $p$ of
$\gcd(N_{13}, N_{40})$. These cases can be easily studied independently (see
the end of the proof). Let us assume from now that it is not the case, i.e.
$p \ne 2$, $3$, $5$, $7$, $13$, $43$, $47$, $89$, $277$ and $523$. Either
$N_{13}\ne0$ or $N_{40}\ne0$.
\begin{itemize}[label=\raisebox{0.25ex}{\tiny$\bullet$}]
\item If $N_{13} \not \equiv  0 \bmod p$, then we have two possibilities\,:
  \begin{itemize}
  \item $4\,a_{13}+15=0$\,, and with a couple of additional resultants, we show
    that $2\,a_{40} + 3 = 0$ when $p \ne 2$, $3$, $13$, $1613$, $3469$ and 6
    other 5-24 digit primes;
  \item $16\,a_{13}^2-60\,a_{13}+225=0$\,, and this implies that $4a_{40}^2 -
    6\,a_{40} + 9 = 0$ when additionally $p \ne 37$, $61$ and $210037$.
  \end{itemize}
\item If $N_{40}\not \equiv 0 \bmod p$, then we have the two
  symmetric possibilities\,:
  \begin{itemize}
  \item $2\,a_{40} + 3 = 0$\,, which implies that $4\,a_{13}+15=0$ when
    additionally $p \ne 53$, $4969$ and 3 new 7-19 digit primes;
  \item $4a_{40}^2 - 6\,a_{40} + 9 = 0$\,, which implies
    $16\,a_{13}^2-60\,a_{13}+225=0$ when $p$ is not one of the primes above.
  \end{itemize}
\end{itemize}
In other words, except modulo the few primes $p$ that we have listed above, we
have either
\begin{multline*}
  a_{31} = -15/4,\ a_{13} = -15/4,\ a_{40} = -3/2 \text{ and }a_{04} = -3/2\text{ or }\\
  a_{04}^2 - 3/2\,a_{04} + 9/4=0,\ a_{40} = -a_{04} + 3/2,\ a_{13} =
  -5/2\,a_{04} + 15/4\text{ and }a_{31} = 5/2\,a_{04}\,.
\end{multline*}
In both cases, this yields quartics $F=0$ which are the union of a cubic curve
and an inflectional tangent, thus GIT-unstable.\smallskip

\noindent$\bullet$ Cases~2.2 and~2.3 are much easier and can be handled again
with integral Gr\"obner basis computations. They do not yield additional
exceptional primes. We omit the details. \smallskip

To end the proof, we have to consider each exceptional prime $p>3$ one by one,
and check over $\bar{\F}_{p}$ the GIT-stability of the quartics $F$ that
cancel the given set of invariants. Since we have to deal now with Gr\"obner
basis computations in finite fields, and no longer over $\Z$, these
computations can be easily done with \Magma.

For $p=2$, a similar work can be done to obtain the same six cases of
Table~\ref{tab:singmodels}, with slightly distinct forms $F$ (see
Table~\ref{tab:singmodelschar2} and ~\cite[File
\texttt{HSOPDetailedProof.txt}]{HSOPList}). Then, a direct Gr\"obner basis
computations easily yields the subset of these models that cancel our set of 7
invariants. The associated curves either have a singular point of order 3, or
are the union of a cubic curve and an inflectional tangent. They thus all are
GIT-unstable and the set is a HSOP.
\end{proof}

\begin{table}[htbp]
  \centering
  \setlength{\tabcolsep}{2pt}
  \renewcommand{\arraystretch}{1.0}
  \begin{tabular}{c|l|l|l|}
    Case & \multicolumn{2}{l|}{\ Condition}& \ $F(x_1,x_2,x_3)$ \\\hline\hline
    0& \multicolumn{2}{l|}{$\Pi$ of order $\geqslant 3$} &
    $
    \begin{array}{c}
      a_{04}\,{x_2}^{4}+a_{13}\,x_1{x_2}^{3}+a_{22}\,{x_1}^{2}{x_2}^{2}+a_{31}\,{x_1}^{3}x_2+a_{40}\,{x_1}^{4}+\\
      ( a_{03}\,{x_2}^{3}+a_{12}\,x_1{x_2}^{2}+a_{21}\,{x_1}^{2}x_2+a_{30}\,{x_1}^{3} )\,x_3
    \end{array}
    $\\\hline\hline
    1.1& \multirow{2}%
    {0.2\textwidth}%
    {$\Pi$ of order 2 \& $x_1=0$ is the only tangent at $\Pi$}%
    & $\begin{array}{c}\ a_{03} = 0\end{array}$ &
    $\begin{array}{l}
       x_1 \: (\, a_{13}\,x_2^3+a_{22}\,x_1\,x_2^2+a_{31}\,x_1^2\,x_2+a_{40}\,x_1^3  +\\
       \ \ \ \ \ \ \ \ \ \ \ \ \ \ \ \ \ \ \ \ \ \ \ \ (a_{21}\,x_1\,x_2+a_{30}\,x_1^2)\,x_3 + x_1\,x_3^2 \,)\\
    \end{array}$%
    \\\cline{1-1}\cline{3-4}
    1.2 & & $\begin{array}{c}\ a_{03} \neq 0\end{array}$ &
    $\begin{array}{c}
       x_2 \: (\, a_{31}\,x_1^3 + x_2^2\,x_3 + x_1\,x_3^2   \,)
    \end{array}$
    \\[0.2cm]\hline\hline
    2.1& \multirow{5}%
    {0.2\textwidth}%
    {$\Pi$ of order 2 \& $x_1=0$, $x_2=0$ are both tangent at $\Pi$}%
    &
    $\begin{array}{c}
      a_{03} \neq 0\\
      \&\,a_{30} \neq 0
    \end{array}$
    &
      $\begin{array}{c}
         x_3 \: (\,  x_2^3 + a_{12}\,x_1\,x_2^2 + (a_{12} + 1)\,x_1^2\,x_2 + x_1^3 + x_1\,x_2\,x_3\,)\\
         \text{ or }\ \ x_3 \: (\,x_2^3 + x_1\,x_2^2 + x_1^2\,x_2 + x_1^3  + x_1\,x_2\,x_3\,)
    \end{array}$
    \\\cline{1-1}\cline{3-4}
    2.2 & &
    $\begin{array}{c}
      a_{03} = 0\\
      \&\,a_{30} \neq 0
    \end{array}$ &
    $\begin{array}{l}
       x_1 \: (\, a_{13}\,x_2^3+a_{22}\,x_1\,x_2^2+a_{31}\,x_1^2\,x_2+a_{40}\,x_1^3 \\
       \ \ \ \ \ \ \ \ \ \ \ \ \ \ \ \ \ \ \ \ \ \ \ \ + (a_{21}\,x_1\,x_2+x_1^2)\,x_3  + x_2\,x_3^2 \,)\\
    \end{array}$
    \\\cline{1-1}\cline{3-4}
    2.2 & &
    $\begin{array}{c}
      a_{03} = 0\\
      \&\,a_{30} = 0
    \end{array}$ &
    $\begin{array}{l}
       x_1 \: (\, a_{13}\,x_2^3+a_{22}\,x_1\,x_2^2+a_{31}\,x_1^2\,x_2+a_{40}\,x_1^3  \\
    \ \ \ \ \ \ \ \ \ \ \ \  \ \ \ \ \ \ \ \ \ \  \  \ \ + a_{21}\,x_1\,x_2\,x_3 + x_2\,x_3^2  \,)
    \end{array}$\\\hline\hline
  \end{tabular}
  \smallskip
  \caption{Singular quartics in characteristic $p=2$}
  \label{tab:singmodelschar2}
\end{table}

We encountered more difficulties to find a HSOP in characteristic
$3$. Especially, it is only after a lengthy computation that we finally manage
to find an invariant, of degree 81, which does not vanish on the orbit of the
GIT-semi-stable quartic $-x_2^3\,x_3+x_1^4+x_1^2\,x_3^2$. This finally yields
a HSOP formed of explicit invariants $I_3^{(3)}$, $I_{27}^{(3)}=I_{27}$,
$J_{27}^{(3)}$, $I_{36}^{(3)}$, $J_{36}^{(3)}$, $I_{54}^{(3)}$ and
$I_{81}^{(3)}$ (we also refer to~\cite[\texttt{DixmierHSOP.txt}]{HSOPList} for
their polynomial expressions in the Dixmier-Ohno invariants).

We do not have a proof that the expressions of $I_{54}^{(3)}$ and
$I_{81}^{(3)}$ in terms of the $15$ coefficients are polynomials with integer
coefficients. Because of the degree of these expressions, it is simply
impossible to even write this expression which has about $2^{56}$ monomials.
We hope that this integrality property can be proved in a near future in a
different way. Meanwhile, supported by extensive experiments that we made on
random and one parameter families of quartics over $\Z$, we state the
following conjecture.
\begin{conjecture}[Integrality hypothesis]\label{conj:hsop3}
  The invariants $I_{54}^{(3)}$ and $I_{81}^{(3)}$ have integral coefficients.
\end{conjecture}

\begin{theorem}\label{th:domod3}
  Let $\O$ be a DVR with residue field of characteristic $p=3$.  Under Conjecture~\ref{conj:hsop3}, the
  invariants $I_3^{(3)}$, $I_{27}^{(3)}=I_{27}$, $J_{27}^{(3)}$, $I_{36}^{(3)}$,
  $J_{36}^{(3)}$, $I_{54}^{(3)}$ and $I_{81}^{(3)}$ form a HSOP over $\O$ for ternary
  quartic forms under the action of $\SL_{3,\O}$.
\end{theorem}
\begin{proof}
  The strategy follows the same lines as  for $p=2$. One shall this time replace
  Table~\ref{tab:singmodels} by  Table~\ref{tab:singmodelschar3} (see ~\cite[File
  \texttt{HSOPDetailedProof.txt}]{HSOPList} for a detailed proof of these cases).
\begin{table}[htbp]
  \centering
  \setlength{\tabcolsep}{2pt}
  \renewcommand{\arraystretch}{1.0}
  \begin{tabular}{c|l|l|l|}
    Case & \multicolumn{2}{l|}{\ Condition}& \ $F(x_1,x_2,x_3)$ \\\hline\hline
    0& \multicolumn{2}{l|}{$\Pi$ of order $\geqslant 3$} &
    $
    \begin{array}{c}
      a_{04}\,{x_2}^{4}+a_{13}\,x_1{x_2}^{3}+a_{22}\,{x_1}^{2}{x_2}^{2}+a_{31}\,{x_1}^{3}x_2+a_{40}\,{x_1}^{4}+\\
      ( a_{03}\,{x_2}^{3}+a_{12}\,x_1{x_2}^{2}+a_{21}\,{x_1}^{2}x_2+a_{30}\,{x_1}^{3} )\,x_3
    \end{array}
    $\\\hline\hline
    1.1& \multirow{2}%
    {0.2\textwidth}%
    {$\Pi$ of order 2 \& $x_1=0$ is the only tangent at $\Pi$}%
    & $\begin{array}{c}\ a_{03} = 0\end{array}$ &
    $\begin{array}{l}
       x_1 \: (\, a_{13}\,x_2^3+a_{22}\,x_1\,x_2^2+a_{31}\,x_1^2\,x_2+a_{40}\,x_1^3  +\\
       \ \ \ \ \ \ \ \ \ \ \ \ \ \ \ \ \ \ \ \ \ \ \ \ (a_{21}\,x_1\,x_2+a_{30}\,x_1^2)\,x_3 + x_1\,x_3^2 \,)\\
    \end{array}$%
    \\\cline{1-1}\cline{3-4}
    1.2 & & $\begin{array}{c}\ a_{03} \neq 0\end{array}$ &
    $\begin{array}{c}
       x_3 \: (\, x_2^3 + a_{30}\,x_1^3 + x_1^2\,x_3   \,)
    \end{array}$
    \\[0.2cm]\hline\hline
    2.1& \multirow{5}%
    {0.2\textwidth}%
    {$\Pi$ of order 2 \& $x_1=0$, $x_2=0$ are both tangent at $\Pi$}%
    &
    $\begin{array}{c}
      a_{03} \neq 0\\
      \&\,a_{30} \neq 0
    \end{array}$
    &
      $\begin{array}{c}
         x_3 \: (\, x_2^3 + x_1^3 + x_1\,x_2\,x_3 \,)
    \end{array}$
    \\\cline{1-1}\cline{3-4}
    2.2 & &
    $\begin{array}{c}
      a_{03} = 0\\
      \&\,a_{30} \neq 0
    \end{array}$ &
    $\begin{array}{l}
       x_1 \: (\, a_{31}\,x_1^2\,x_2 + a_{40}\,x_1^3 + x_1^2\,x_3 + x_2\,x_3^2  )
    \end{array}$
    \\\cline{1-1}\cline{3-4}
    2.2 & &
    $\begin{array}{c}
      a_{03} = 0\\
      \&\,a_{30} = 0
    \end{array}$ &
    $\begin{array}{c}
       x_1\:(\, a_{31}\,x_1^2\,x_2 + a_{40}\,x_1^3 + x_2\,x_3^2 \,)\\
       \text{ or }\ \ x_2\:(\, a_{04}\,x_2^3 + a_{13}\,x_1\,x_2^2 + x_1\,x_3^2 \,)
    \end{array}$\\\hline\hline
  \end{tabular}
  \smallskip
  \caption{Singular quartics in characteristic $p =3$}
  \label{tab:singmodelschar3}
\end{table}

\end{proof}

\begin{remark}\label{rmk:nhspecial} With the same assumptions as
  Theorem~\ref{th:gnhr}, the invariants of the special fiber $\mathcal{C}_k$
  are the reduction modulo $\pi$ of a minimal representative of
  $\Iv(F)$. Hence, if an algorithm to reconstruct from the invariants is
  available, one can get an equation for $ \mathcal{C}_k$. This is in
  particular the case for generic plane quartics in characteristic $p>7$ when
  $\DOv \subset \Iv$ using \cite{LRS16}.
\end{remark}

\begin{example} \label{ex:cm} Let us consider the quartic $X_1 : F_1=0$ with
\begin{multline*}
  F_1 = - 4169 x_1^4 - 956 x_1^3 x_2 + 7440 x_1^3 x_3 + 55770 x_1^2 x_2^2 +
  43486 x_1^2 x_2 x_3 + 42796 x_1^2 x_3^2 - 38748 x_1 x_2^3 - 30668 x_1 x_2^2
  x_3\\ + 79352 x_1 x_2 x_3^2 - 162240 x_1 x_3^3 + 6095 x_2^4 + 19886 x_2^3
  x_3 - 89869 x_2^2 x_3^2 - 1079572 x_2 x_3^3 - 6084 x_3^4
\end{multline*}
from \cite[Sec.5]{KLLRSS17}. The values at $F_1$ of its Dixmier-Ohno
invariants are
\begin{align*}
I_3(F_1) &= 2^{-2} \cdot 3^{-2} \cdot 5 \cdot 13^2 \cdot  43 \cdot 108879238253 \\
I_6(F_1) &= 2^{-6} \cdot 3^{-6} \cdot 5 \cdot 13^4 \cdot 33879904575927947128535137 \\
& \ldots \\
J_{21}(F_1) &=  2^{-18} \cdot 3^{-16} \cdot 5 \cdot 7 \cdot 11\cdot 13^{14} \cdot 89 \cdot 11383 \cdot \textrm{huge prime} \\
I_{27}(F_1) &=  - 2^{30}  \cdot 5^{12} \cdot 7^9 \cdot 13^{18} \cdot 37^{14} \cdot 15187^{14}.
\end{align*}
For any prime  $p$ not dividing $D_{27}(F_1)=2^{40} \cdot I_{27}(F_1)$, $X_1$ has good quartic reduction.
Let us now compute the valuation at $p$ of the various HSOPs we found and the normalized valuation of  $D_{27}(F_1)$ with respect to the HSOP.\smallskip
\begin{center}
\begin{tabular}{|c|c|c|}
\hline
$p$ &  valuations at $p$ &  normalized val. of $D_{27}(F_1)$ \\
\hline
$2$ & $2, 4, 6, 2, 11, 0, 70$ & ${70}/{27}$   \\
$5$ & $1, 1, 0, 2, 1, 2, 12 $ & ${12}/{27}$  \\
$7$ & $0, 0, 0, 0, 0, 0, 9$ & ${9}/{27}$ \\
$13$ &  $2, 4, 6, 8, 10, 12, 18$ & $\frac{18}{27} -  \min\left(\frac{2}{3},\frac{4}{6},\frac{6}{9},\frac{8}{12},\frac{10}{15},\frac{12}{18},\frac{18}{27}\right)=0$  \\
$37$ & $0, 0, 0, 0, 0, 0, 14$ & ${14}/{27}$ \\
$15187$  & $0, 0, 0, 0, 0, 0, 14$ & ${14}/{27}$\\
\hline
\end{tabular}
\end{center}
\smallskip

By using Theorem~\ref{th:GNHR}, it is easy to see that among these primes,
$X_1$ has potentially good quartic reduction if and only if $p=13$.  A minimal
representative of $\DOv(F_1)$ reduces modulo $13$ to
$( 9, 2, 10, 5, 12, 6, 1, 9, 3, 3, 10, 8, 11 )$. Using~\cite{LRS16}, we find
that they are the invariants of the smooth plane quartic
\begin{math}
  x_1^3\,x_3 + x_1^2\,x_2^2 + x_2^3\,x_3 + 4\,x_3^4=0\,.
\end{math}
We will resume in Example~\ref{ex:X1} the study of $X_1$ and see for which
primes among $p=2,5,7,37$ and $15187$, the reduction is not potentially good.
\end{example}

\begin{remark} \label{rem:nonlift}
  Finding a set of generators for the algebra of invariants in small
  characteristics $p$ to replace or complete the Dixmier-Ohno invariants is,
  as far as we know, an open problem. Notice that the situation is more subtle
  than for the HSOP as there may be invariants which do not come from
  reduction of an invariant in characteristic $0$. Over $\bar{\F}_3$, the
  degree-6 homogeneous component is not of dimension $2$ as in characteristic
  $0$, but of dimension $3$, generated by $i_3^2$, $i_6$ and $i_6'$ where
  \begin{displaymath}
 i_3 =  3^3\ I_3,\quad   i_6 = 3^2\ (18\,{I_{6}}-{{I_{3}}}^{2})\ \text{ and }\
 i'_6 = 3^5\ (9\,{{{I_{9}}}-{{I_{3}}}^{3})\,/\,{{i_{3}}}}.
  \end{displaymath}
  Notice that these expressions must be understood as the reduction of the
  polynomial expressions in the coefficients of a general ternary quartic form
  over $\Z$.
\end{remark}

\subsection{The case of Picard curves} \label{sec:Picard} As an other
illustration, let us look at the well-studied family of Picard
curves~\cite{BBW17, Hol, KLS18, kowe, LaSo}. Since a Picard curve $C/K$ cannot
be hyperelliptic (see for instance \cite[p.13]{harris}), the reduction of the
stable model of $C$ is either a non-hyperelliptic (Picard) curve or
singular. Thus, we can get a complete characterization of the type of
potential reduction with Theorem~\ref{th:GNHR}. In the present section, we go
one step further and translate these conditions on the valuation of the
involved invariants in human readable conditions. This study is split into
several sub-cases. In order to shorten it, we exclude the case where the
characteristic of $K$ is $3$ (but the residue field can be of characteristic
$3$).

After a possible finite extension of $K$, we extend to $\O$ the cyclic
covering to $\P^1/K$, and we can even assume that $C$ has a model
$\Cc_0 : F_0=0$ over $\O$ where
\begin{align}\label{eq:generalPicard}
\begin{split}
  F_0&=-x_2^3x_3+x_1^4+bx_1^2x_3^2+c x_1x_3^3+d x_3^4,\; \textrm{when }
  \textrm{Char}(K)\ne 2,\\%
  F_0&=-x_2^3x_3+x_1^4+ax_1^3 x_3+bx_1^2x_3^2+cx_1 x_3^3+d x_3^4,\;
  \textrm{when } \textrm{Char}(K)=2,
\end{split}
\end{align}
with $a,b,c,d \in \O$.  Note that we could for instance let $d=0$ in the
second model but the transformation to get it may require a larger extension
of $K$ and we will not need this to prove our result.

To simplify the expressions, we introduce three invariants of binary quartic
forms
${a_4}\,{x}^{4}+{a_3}\,{x}^{3}\,z+{a_2}\,{x}^{2}\,{z}^{2}+{a_1}\,x\,{z}^{3}+{a_0}\,{z}^{4}$
under the action of $\SL_{2,\O}$, namely
\begin{eqnarray*}
  q_2 &=& 12\,{a_0}\,{a_4}-3\,{a_1}\,{a_3}+{{a_2}}^{2}\,,\\
  q_3 &=& 72\,{a_0}\,{a_2}\,{a_4}-27\,{a_0}\,{{a_3}}^{2}-27\,{{a_1}}^{2}{a_4}+9\,{a_1}\,{a_2}\,{a_3}-2\,{{a_2}}^{3},
\end{eqnarray*}
and the discriminant $D_6$ which satisfies the relation
$3^3 \cdot D_6 = 4\,q_2^3 - q_3^2$\,.
Let us start with the case $p \ne 2$.

\begin{theorem} \label{ex:picard} Let $K$ be a discrete valuation field with
  valuation $v$, valuation ring $\O$ and a uniformizer $\pi$. Let
  $k=\Oc/\langle \pi \rangle$ be the residue field of characteristic $p \ne 2$
  and $3$.  The curve $\Cc_0 : -x_2^3x_3+G_0(x_1,x_3)=0$ where
  $G_0=x^4+ b x^2 z^2+cx z^3+d z^4=0$ has potentially good quartic reduction
  if and only if
  \begin{equation} \label{eq:valuations}
    \min(6 v(b),3 v(d)) \geq v(D_6(G_0))\,.
  \end{equation}
  A stable model $\Cc$ is given by
  $\mathcal{C}:\,-x_2^3x_3+G(x_1,x_3)=0$ where $$G=x_1^4+b/\mathfrak{p}^6\,x_1^2x_3^2+c/\mathfrak{p}^9\,x_1x_3^3+d/\mathfrak{p}^{12}\,x_3^4= 0,$$ $\mathfrak{p}=\pi^{v(D_6(G_0))/36}$, and the map
  $\mathcal{C}\rightarrow\mathcal{C}_0:\,(x_1:x_2:x_3)\mapsto(\mathfrak{p}^3x_1:\mathfrak{p}^4x_2:x_3)$.
\end{theorem}
\begin{proof}
 We have
$$v_6:=D_6(G_0)= 2^8 \cdot 3^3 d^3- 2^7 \cdot 3^3 b^2 d^2+ 2^4 \cdot 3^3  b^4
d+c^2 (2^4 \cdot 3^5 b d -3^6 c^2-2^2 \cdot 3^3 b^3).$$ Assume the
conditions~\eqref{eq:valuations}, we first want to prove that $G$ is defined
over $\O$, which boils down to $4 v(c) \geq v_6$.  From the expression of
$D_6(G_0)$ above, the conditions imply
$$2 v(c) + v(2^4 \cdot 3^5 b d -3^6 c^2-2^2 \cdot 3^3 b^3) \geq v_6.$$ If
$v(c) < v_6/4$, then $v(2^4 \cdot 3^5 b d -3^6 c^2-2^2 \cdot 3^3 b^3) < v_6/2$
and we get a contradiction.
Hence, the equation of $\mathcal{C}$ has integral coefficients and its discriminant $$D_{27}(-x_2^3x_3+G) = 3^9  D_6(G)^2 = 3^9 D_6(G_0)^2/\mathfrak{p}^{72}$$ has valuation $0$. $\Cc$ is therefore a model of $\Cc_0$ with good quartic reduction.\\
Conversely, let us assume that $\mathcal{C}_0$ has potentially good quartic
reduction. If we use the same notation for the invariants and their values at
$F_0$ or $G_0$, a quick computation yields
\begin{equation}\label{eq:2}
  \begin{array}{lcl}
    0 & = & I_3 = I_6 = I_{12} = J_{12} = I_{15} = J_{15} = I_{21} = J_{21}\,,\\
    I_9   & =& 2^{-12} \cdot 3^{-4}\ (8\,b\,{q_3}+81\,{{q_2}}^{2})\,,\\
    J_9   & =& 2^{-12} \cdot 3^{-4}\ (16\,b\,{q_3}+27\,{{q_2}}^{2})\,,\\
    I_{18} &=& 2^{-23} \cdot 3^{-6}\ (108\,{b}^{2}\,{{q_2}}^{3}+33\,b\,{q_3}\,{{q_2}}^{2}+8\,{{q_2}}^{4}-54\,{D_6}\,{q_2})\,,\\
    J_{18} &=& 2^{-23} \cdot 3^{-7}\ (36\,{b}^{2}\,{{q_2}}^{3}+51\,b\,{q_3}\,{{q_2}}^{2}+16\,{{q_2}}^{4}-108\,{D_6}\,{q_2})\,,\\
    I_{27} &=& 2^{-40} \cdot 3^9 \ \ {D_6}^2\,.
  \end{array}
\end{equation}
When $p \ne 5$, by the criterion of Theorem~\ref{th:GNHR} used with the $\DO$,
we get that $v(I_9) \geq v(I_{27})/3$ and $v(J_9) \geq v(I_{27})/3$. Since
$q_2^2 = 2^{12}\cdot3\cdot5^{-1}\,(2\,I_9-J_9)$, we get that
$v(q_2) \geq v(D_6)/3$.  Furthermore since $3^3\,D_6 = 4\,q_2^3 - q_3^2$,
$v(q_3) \geq v(D_6)/2$.  Specifically, at least one of these two inequalities
is an equality.

If for instance $2\,v(\,q_3\,) = v(D_6)$, since
$b\,q_3 = 2^{9}\cdot3^{4}\cdot5^{-1}\,(3\,J_9-I_9)$,
$v(3\,J_9-I_9)/3 \geq v(I_{27})$ implies $3 v(b\,q_3) \geq 2 \,v(D_6)$, or
equivalently $6\,v(b) \geq 4\,v(D_6) - 6v(q_3) = v(D_6)$\,.  Moreover, since
$q_2=b^2+12 d$ and $v(q_2) \geq v(D_6)/3$, we get $v(d) \geq v(D_6)/3$.

Otherwise, if $3\,v(\,q_2\,) = v(D_6)$, consider the equality
\begin{displaymath}
  2^{21} \cdot 3^4 \cdot 5^{-1} \cdot (2\,I_{18}-3\,J_{18} +
  (I_{9}-3\,J_9)\,(2\,I_9-J_9)/20 ) = b^2\,q_2^3.
\end{displaymath}
We find $3\,v(b^2\,q_2^3) \geq 2\,v(I_{27})$, and thus
$6\,v(b) \geq v(D_6)$ as before and we finish the proof in the same way for $d$.\smallskip

When $p=5$, we use another set of invariants from Theorem~\ref{th:domodp}, in
particular $J_9^{(5)} = (J_9 - 2\,I_9)/5$. Similarly \begin{math} q_2^2 =
  -2^{12} \cdot 3\:J_9^{(5)}\text{ and }b\,q_3 =2^{9}\cdot3^{4}
  \:(\,2\,\,J_9^{(5)} + I_9^{(5)}\,)\,.
\end{math}
We can finish the proof in the same way.
\end{proof}

We now turn to the case $p=2$ and split it into two sub-cases according to $\textrm{Char}(K)=2$ or not.
\begin{theorem} \label{ex:picard2} Let $K$ be a discrete valuation field with
  valuation $v$, valuation ring $\O$ and a uniformizer $\pi$. Let
  $k=\Oc/\langle \pi \rangle$ be the residue field of characteristic $p=2$. %
\begin{itemize}
\item Assume that $\Char(K) = 0$. The curve
  $\Cc_0 : -x_2^3x_3+G_0(x_1,x_3)=0$ where
  $G_0=x^4+b x^2z^2+c x z^3+d x^4$ has potentially good quartic reduction if
  and only if
  \begin{equation}\label{eq:valuations2}
    6\,v(2^3 b) \geq v(D_6(G_0))\ \text{ and }\ %
    3\,v(q_2) \geq v(D_6(G_0))\,.
  \end{equation}

\item Assume that $\Char(K) =2$. The curve $\Cc_0 : -x_2^3x_3+G_0(x_1,x_3)=0$
  where $G_0=x^4+ax ^3 z+bx ^2 z^2+cx  z^3+ d z^4$ has potentially good quartic
  reduction if and only if
  \begin{equation}\label{eq:valuations2.2}
    12\,v(a) \geq v(D_6(G_0))\ \text{ and }\ %
    3\,v(q_2) \geq v(D_6(G_0))\,.
  \end{equation}
\end{itemize}
\end{theorem}
\begin{proof}
  We start with the case $\Char(K) = 0$.  We therefore use
  $G_0=x^4+b x^2 z^2+c x z^3+d z^4=0$ and the HSOP
  ${\DO^{(2)}} = \{ I_3^{(2)},\, I_{12}^{(2)},\,\ldots, I_{27}^{(2)}\}$, given
  in Theorem~\ref{th:domodp} for $p=2$. If we use the same notation for the
  invariants and their values at $F_0$ or $G_0$, a quick computation
  yields \begin{displaymath}
    \begin{array}{rcl}
      0 &=& I_3^{(2)} = I_{12}^{(2)} = I_{15}^{(2)} =  I_{24}^{(2)}\,,\\
      I_9^{(2)} &=& 2^{3}\cdot3^{4}\:b\,q_3 + 3^{8}\:q_2^2\,,\\
      I_{18}^{(2)} &=& 2^{6}\cdot3^{14}\:b^2\,q_2^3 + 2^{4}\cdot3^{12}\cdot11\:b\,q_3\,q_2^2 + 2^{5}\cdot3^{11}\:q_3^2\,q_2\,,\\
      I_{27}^{(2)} &=&   D_{27} = 3^9 D_6^2 \,.
    \end{array}
\end{displaymath}
Let us first assume the conditions~\eqref{eq:valuations2}. Since
$3^3 \cdot D_6 = 4 q_2^3-q_3^2$, they imply that $v(q_3)=v(D_6)/2$. Hence
$v(I_9^{(2)}) \geq v(D_6^2)/3$, $v(I_{18}^{(2)}) \geq 2/3 v(D_6^2)$ and
therefore $v_{\DO^{(2)}}(\,D_{27}\,)=0$. By Theorem~\ref{th:GNHR}, $\Cc_0$ has
potentially good quartic reduction.

Now, assume that $\Cc_0$ has potentially good quartic reduction. As in the
proof of Theorem~\ref{ex:picard}, a trick to shortcut the computations with
the valuations is to introduce more invariants than the ones in the HSOP. Here
we make use of the invariant $J_9^{(2)} = 2^{12} \cdot 3^7 J_9$ which is
defined over $\O$. From equation~\eqref{eq:2}, it is easy to check that
 $$q_2^2 = 3^{-7} \cdot 5^{-1}\,(\,2\,I_9^{(2)} - 3\,J_9^{(2)}\,)\text{ and
  }2^3\,b\,q_3 = 3^{-4} \cdot 5^{-1}\,(\,3^2\,\,J_9^{(2)} -
  I_9^{(2)}\,)\,.
$$
Theorem~\ref{th:GNHR} applied to this broader set of invariants translates
into $$v(q_2^2) = v(2 I_9^{(2)} - 3 J_9^{(2)}) \geq v(D_{27}) =v(D_6^2)/3$$
and so we get the second condition. Similarly, $v(2^3\,b q_3) \geq 2
v(D_6)/3$. Since $v(D_6)=v(4 q_2^3-q_3^2)$ implies as before
$v(q_3)=v(D_6)/2$, we get $v(2^3\,b) \geq v(D_6)/6$.\\

It remains to consider the case $\Char(K)=2$ with the model
$G_0=x^4+ax^3z +bx^2 z^2+cx z^3+d z^4$.  We get
\begin{displaymath}
  \begin{array}{rcl}
    0 &=& I_3^{(2)} = I_{12}^{(2)} = I_{15}^{(2)} =  I_{24}^{(2)}\,,\\
    I_9^{(2)} &=& a^2 q_3 + q_2^2, \\
    I_{18}^{(2)} &=& a^4 q_2^3, \\
    I_{27}^{(2)} &=&   D_{27} = D_6^2 = q_3^4\,.
  \end{array}
\end{displaymath}
When the conditions~\eqref{eq:valuations2.2} are satisfied, it is easy to
check that $v(I_9^{(2)}) \geq v(D_6^2)/3$ and
$v(I_{18}^{(2)}) \geq 2v(D_6^2)/3$ so $\Cc_0$ has potentially good quartic
reduction. Conversely, making use of the invariant $J_9^{(2)} = q_2^2$, we
find $v(q_2) \geq v(D_6)/3$ and then since $v(I_9^{(2)}) \geq v(D_6^2)/3$ and
$v(q_3)=v(D_6)/2$, we get $v(a) \geq v(D_6)/12$.
\end{proof}

Finally we look at the case $p=3$ (and $\Char(K)=0$).
\begin{theorem} \label{ex:picard3} Let $K$ be a discrete valuation field of
  characteristic 0 with valuation $v$, valuation ring $\O$ and a uniformizer
  $\pi$. Let $k=\Oc/\langle \pi \rangle$ be the residue field of
  characteristic $p=3$. %
  Under Conjecture~\ref{conj:hsop3}, the curve
  $\Cc_0 : -x_2^3x_3+G_0(x_1,x_3)=0$ where $G_0=x^4+ b x ^2 z^2+c x z^3+d z^4$
  has potentially good quartic reduction if and only if
  \begin{equation}\label{eq:valuations3}
    6\,v(3^{5/4}\,b) \geq v(D_6(G_0))\ \text{ and }\ %
    2\,v(q_3) \geq v(3^5\,D_6(G_0))\,.
  \end{equation}
\end{theorem}
\begin{proof}
  We use the HSOP
  ${\DO^{(3)}} = \{ I_3^{(3)},\, I_{27}^{(3)},\,\ldots, I_{81}^{(3)}\}$, given
  in Theorem~\ref{th:domod3}. In the case of the model $\Cc_0$ we get in
  particular $I_3^{(3)} = 0$ and $I_{27}^{(3)} = D_{27} = 3^9 D_6^2$\,.  Let
  us define $\nu := v(D_6(Q_0))\,/\,6$\,.

  Let us first assume the conditions~\eqref{eq:valuations3}, \textit{i.e.}
  $q_3=O(3^{{5}/{2}}\pi^{3\nu})$. Since $3^3 \cdot D_6 = 4 q_2^3-q_3^2$, it
  implies that $v(q_2)=v(3\,\pi^{2\nu})$ and
  $D_6 = 4\, (q_2/3)^3 + O( 3^2\pi^{6\nu})$.  Hence, a quick computation shows
  that
  \begin{dgroup*}[style={\footnotesize},spread={-2pt}]
    \begin{dmath*}
      J_{27}^{(3)} = %
      O(3^{14.5}{\pi}^{9\,\nu})\cdot{b}^{3}\,+
      O(3^{13}{\pi}^{10\,\nu})\cdot{b}^{2}\,+
      O(3^{11.5}{\pi}^{11\,\nu})\cdot{b}\,+
      O(3^{9}{\pi}^{12\,\nu})\,,
    \end{dmath*}
    \begin{dmath*}
      I_{36}^{(3)} = %
      O(3^{17}{\pi}^{12\,\nu})\cdot{b}^{4}\,+
      O(3^{17.5}{\pi}^{13\,\nu})\cdot{b}^{3}\,+
      O(3^{16}{\pi}^{14\,\nu})\cdot{b}^{2}\,+
      O(3^{14.5}{\pi}^{15\,\nu})\cdot{b}\,+
      O(3^{13}{\pi}^{16\,\nu})\,,
    \end{dmath*}
    \begin{dmath*}
      J_{36}^{(3)} = %
      O(3^{17}{\pi}^{12\,\nu})\cdot{b}^{4}\,+
      O(3^{17.5}{\pi}^{13\,\nu})\cdot{b}^{3}\,+
      O(3^{16}{\pi}^{14\,\nu})\cdot{b}^{2}\,+
      O(3^{14.5}{\pi}^{15\,\nu})\cdot{b}\,+
      O(3^{13}{\pi}^{16\,\nu})\,,
    \end{dmath*}
    \begin{dmath*}
        I_{45}^{(3)} = %
        O(3^{22.5}{\pi}^{15\,\nu})\cdot{b}^{5}\,+
        O(3^{21}{\pi}^{16\,\nu})\cdot{b}^{4}\,+
        O(3^{20.5}{\pi}^{17\,\nu})\cdot{b}^{3}\,+
        O(3^{20}{\pi}^{18\,\nu})\cdot{b}^{2}\,+
        O(3^{17.5}{\pi}^{19\,\nu})\cdot{b}\,+
        O(3^{16}{\pi}^{20\,\nu})\,,
    \end{dmath*}
    \begin{dmath*}
        I_{81}^{(3)} = %
        O(3^{40.5}{\pi}^{27\,\nu})\cdot{b}^{9}\,+
        O(3^{38}{\pi}^{28\,\nu})\cdot{b}^{8}\,+
        O(3^{37.5}{\pi}^{29\,\nu})\cdot{b}^{7}\,+
        O(3^{36}{\pi}^{30\,\nu})\cdot{b}^{6}\,+
        O(3^{34.5}{\pi}^{31\,\nu})\cdot{b}^{5}\,+
        O(3^{34}{\pi}^{32\,\nu})\cdot{b}^{4}\,+
        O(3^{32.5}{\pi}^{33\,\nu})\cdot{b}^{3}\,+
        O(3^{31}{\pi}^{34\,\nu})\cdot{b}^{2}\,+
        O(3^{29.5}{\pi}^{35\,\nu})\cdot{b}\,+
        O(3^{29}{\pi}^{36\,\nu})\,.
      \end{dmath*}
    \end{dgroup*}
    Since $b=O(3^{-1.25}\,\pi^v)$, we obtain
    $J_{27}^{(3)} = O(3^9\,\pi^{12\nu})$, %
    $I_{36}^{(3)} = O(3^{13}\,\pi^{16\nu})$, %
    $J_{36}^{(3)} = O(3^{12}\,\pi^{16\nu})$, %
    $I_{45}^{(3)} = O(3^{16}\,\pi^{20\nu})$ %
    and %
    $I_{81}^{(3)} = O(3^{28}\,\pi^{36\nu})$\,. %
    By Theorem~\ref{th:GNHR}, $\Cc_0$ has potentially good quartic
    reduction.

    Conversely, let us assume that $\mathcal{C}_0$ has potentially good
    quartic reduction. So $I_{81}^{(3)} =
    O(3^{27}\,\pi^{36\nu})$. Since, $I_{81}^{(3)} \equiv q_2 \bmod 3^4$, we
    have $q_2 = O(3)$. Furthermore
    $I_{81}^{(3)}\equiv 3^{18}\,(D_6-4(q_2/3)^3)^6 \bmod 3^{20}$.  We thus
    must have $D_6 = 4\, (q_2/3)^3 + O( 3^2)$. Since
    $3^3 \cdot D_6 = 4 q_2^3-q_3^2$, it further implies $q_2=O(3\,\pi^{2\nu})$ and
    $q_3=O(3^{{5}/{2}}\pi^{3\nu})$.  By hypothesis, we also must have
    $J_{27}^{(3)} = O(3^{9}\,\pi^{12\nu})$, and this yields
    $b=O(3^{-1.25}\,\pi^v)$ too.
\end{proof}

\begin{remark} \label{rem:Picardmod3}In \cite{BKKSW}, the authors describe the
  arithmetic properties of models of Picard curves when $\Char(K) \ne 3$.
\end{remark}

\section{Characterization of potentially good hyperelliptic reduction}

\subsection{Characterization of \PHM{} in terms of invariants} \label{sec:comphyp}

Let $K$ be a discrete valuation field with valuation $v$, valuation ring $\O$
and a uniformizer $\pi$. We now restrict to $k=\Oc/\langle \pi \rangle$ of
characteristic $p =0 $ or $p>7$.  We start with the following observation.

\begin{lemma} \label{lem:test-conic} Let $Q \in \O[x_1,x_2,x_3]$ be a
  quadratic ternary form. Then the Dixmier-Ohno invariants of $Q^2$ are
  \begin{multline*}
    \DOv(Q^2)= \left(I_3 , \frac{1}{2^2\cdot3^2\cdot5} I_3^2 ,
      \frac{7^2}{2^2\cdot3^2} I_3^3 , \frac{7^2}{2^2\cdot3\cdot5} I_3^3 ,
      \frac{7^3}{2^2\cdot3^4\cdot5} I_3^4 , \frac{7^2}{2^2\cdot3^2} I_3^4 ,
      \frac{5\cdot7^3}{2^4\cdot3^5} I_3^5 , \right.
    \\
    \left. \frac{7^3}{2^4\cdot3^2\cdot5^2} I_3^5 , \frac{7^4}{2^4\cdot3^5}
      I_3^6 , \frac{7^4}{2^4\cdot3^3\cdot5^2} I_3^6 ,
      \frac{7^4}{2^2\cdot3^4\cdot5} I_3^7, \frac{7^4}{2^4\cdot3^2\cdot5} I_3^7
      , 0\right)
  \end{multline*}
  where $I_3 = 5/(2^2\cdot3^2) \cdot D_3(Q)^2$ where $D_3$ is the discriminant of ternary quadratic forms.
\end{lemma}
In particular, if $Q$ is non-degenerate, we have
\begin{small}
  \begin{equation}\label{eq:1}
    \DO(Q^2) = \left(1: \frac{1}{180}: \frac{49}{36}: \frac{49}{60}: \frac{343}{1620}: \frac{49}{36}: \frac{1715}{3888}:
      \frac{343}{3600}: \frac{2401}{3888}: \frac{2401}{10800}: \frac{343}{1620}:
      \frac{2401}{720}:  0\right).
  \end{equation}
\end{small}
\begin{remark} \label{rem:kollar} With Lemma~\ref{lem:test-conic}, we are now
  in position to determine the reduction type \emph{over $K$} of a smooth
  plane quartic $F=0$ (at least for any DVR $\O$ for which one can solve
  polynomial equations over $k$). The extra ingredients needed come from the
  methods developed in \cite{kollar} and forthcoming
  \cite{elsenhans-stoll}. Indeed, their algorithms compute the finite list $S$
  of $\PGL_3(\O)$-equivalent classes of forms over $\O$ which are
  $\PGL_3(K)$-equivalent to $F$ and with a discriminant of minimal
  valuation. Now, the discussion splits into two cases. Clearly, one has good
  quartic reduction if and only if there is a form in $S$ with a discriminant
  of valuation $0$ (and in this case one has actually $\# S=1$). \\Secondly,
  one has good hyperelliptic reduction if and only if at least one form
  $F_0 \in S$ satisfies the conditions of Proposition~\ref{2->1}. The reverse
  implication is straightforward. For the direct implication, we know that the
  curve admits a good toogle model $F_0=0$ over $K$
  (Theorem~\ref{th:goodmodel}).  Its discriminant is not of valuation $0$ but
  $I_3(F_0)$ is, as we have just seen\footnote{With the notation of
    Theorem~\ref{th:domodp}, the same property holds replacing $I_3$ by
    $I_{24}^{(2)}$ (resp. $I_3^{(3)}$, resp. $J_{15}^{(5)}$) in characteristic
    $2$ (resp. $3$, resp. $5$).}. This means that the valuation of the
  discriminant of $F_0$ is minimal and therefore $F_0$ belongs to the class of
  a form in $S$, which is also a good toogle model (the properties of being a
  good toogle model do not change by the action of $\PGL_3(\O)$). Unlike the
  case of good quartic reduction, $S$ may contain several forms having either
  bad or good hyperelliptic reduction. For instance $F/\Q_{\pi}=0$ with
  $\pi \in \Z_{>5}$ where $F=(x_2^2-4 x_1 x_3)^2 + \pi^4 (x_1^4+x_3^4)$ has
  good hyperelliptic reduction whereas $F(\pi^2 x_1,\pi x_2,x_3)/\pi^4=0$ has
  bad reduction and both have a discriminant with minimal valuation.
\end{remark}

\begin{proposition} \label{prop:MHR} The curve $C$ potentially admits a \PHM{}
  if and only if any minimal representative of $\DO(F)$ reduces modulo $\pi$
  to the invariants in equation~\eqref{eq:1}.
\end{proposition}
\begin{proof}
  The direct implication is clear.  For the converse, using
  Proposition~\ref{prop:ssred}, and after a possible extension, we can assume
  that we start with a model $C :F=0$ with $F \in \O[x_1,x_2,x_3]$ and that
  $\bar{F}=0$ is a GIT-semi-stable quartic with minimal orbit.  Since
  $\bar{F}=0$ has the same invariants as the square of a non-degenerate
  quadric, which has an infinite automorphism group, it is not GIT-stable.
  The quartic $\bar{F}=0$ is therefore in $X^{ss} \setminus X^s$.  By
  \cite[p.49]{mum-ens} and \cite[Lem.1.4.iii]{artebani}, after possibly a
  finite extension of $K$, $\bar{F}$ is either equivalent to the square of the
  non-degenerate quadratic form $Q_0=(x_2^2-4x_1 x_3)$ or to the product of
  $Q_0$ with a quadratic form $Q$ such that $Q_0=0$ is tangent to $Q=0$ (the
  tangent line being, say, $x_1=0$). In the first case, this means exactly
  that $F$ is a \PHM{}. In the second case, we have
    $$\bar{F}=(x_2^2- 4 x_1 x_3) (a x_1^2+b x_1 x_2+c x_1 x_3+d x_2^2) \text{
    with } (a,b,c,d) \in k^4.$$
  Moreover, since  $\DO(\bar{F})=\DO(Q_0^2)$,
  a computation shows that $c \ne 0$ and $c=-4d$. Hence we
  can rewrite (up to a multiplicative constant)
  $\bar{F}=Q_0^2 + Q_0 x_1 (ax_1+bx_2)$. We therefore see that
  $$F=Q_0^2 + Q_0 x_1 (ax_1+bx_2) + \pi G$$ with $G \in \Oc[x_1,x_2,x_3]$.
  Finally the change of
    variables
    $(x_1,\,x_2,\,x_3) \mapsto (\pi^{1/4}\,x_1\,,\,\pi^{1/8}\,x_2\,,\,
    x_3),$
transforms $F$ into
  \begin{displaymath}
    \pi^{1/2} \cdot \left(Q_0^2 + \pi^{1/8} \underbrace{\left(Q_0 x_1 (a \pi^{1/8} x_1+b x_2) + \pi^{3/8} G(\pi^{1/4} x_1,\pi^{1/8} x_2,x_3)\right)}_{G_0}\right).
  \end{displaymath}
Hence, over an extension of $K$, we find the required model
$F_0= Q_0^2 + \pi^{1/8} G_0$.
\end{proof}
Since having potentially good hyperelliptic reduction implies the existence of
a \PHM{}, Proposition~\ref{prop:MHR} is often strong enough to rule out this
scenario, as the following example shows.
\begin{example} \label{ex:X1} Let $K=\Q_{11}$ and consider the case
  $X_{12}/K : F=0$ from \cite[Sec.5]{KLLRSS17}. The valuations of $\DOv(F)$
  are $(0, 0, 0, 1, 0, 0, 2, 0, 1, 0, 0, 0, 9 )$ and hence $\DOv(F)$ is a
  minimal representative. Since $v_{\DO}(J_9(F))=v(J_9(F))>0$, the list
  $\DOv(F)$ cannot reduce modulo $\pi$ to equation~\eqref{eq:1}. Therefore
  $X_{12}$ does not have potentially good reduction.

  For the excluded characteristics $p =5$ or $7$ (and conjecturally $3$), one
  could replace the invariants in $\DOv$ by the HSOPs of
  Theorems~\ref{th:domodp} and \ref{th:domod3} and still get a
  characterization of having potentially a toggle model. For $p=2$, we get a
  necessary condition for the curve to have potentially good hyperelliptic
  reduction. Let us illustrate this by resuming the case of $X_1$ from
  Example~\ref{ex:cm}. We see for instance that the valuations at $5$ of the
  HSOP for $(x_2^2-4 x_1x_3)^2$ are $(1, 1, 1, 3, 0, 6, \infty)$ whereas the
  ones of $X_1$ are $(1, 1, 0, 2, 1, 2, 12)$. We can conclude that $X_1$ has
  not potentially good reduction at $5$ and also at $7$ with the same
  arguments. However at $2$, the reduction of the HSOP of $X_1$ is the same as
  the ones of $(x_2^2-x_1 x_3)^2$ and we cannot conclude on wether $X_1$ has
  potentially good hyperelliptic reduction at $2$ or does not have potentially
  good reduction. Notice that we slightly change the conic here so it remains
  non-degenerate modulo $2$ and apply Theorem~\ref{th:is2}. We shall see how
  to deal with the prime $2$ in Example~\ref{ex:final}.
\end{example}

\subsection{The equivariant map $b_8$ and proof of Theorem~\ref{th:main}}
\label{sec:char-hyper-reduct}\label{sec:dixm-ohno-invar1}

The equivalence in Theorem~\ref{th:main} is more complicated to establish than
the one in Theorem~\ref{th:GNHR} as we must link the world of genus 3
non-hyperelliptic curves (as ternary quartic forms) and hyperelliptic curves
(as binary octics). The starting point is a beautiful result from
representation theory we will recall now.

Let $R$ be an integral domain and let $W$ (resp. $V$) be free $R$-modules of
rank $n=2$ (resp. $n=3$). We identify binary (resp. ternary) forms with
coefficients in $R$ of degree $d \geq 0$ with elements in $\Sym^d(W^{\vee})$
(resp. $\Sym^d(V^{\vee})$). The usual action of $\GL_n(R)$ on the left for $W$
(resp. $V$) induces a left action on $\Sym^d(W)$ (resp. $\Sym^d(V)$) and a
right action on the forms, which we denote in both cases with a dot.  An
identification of $V$ with $\Sym^2(W)$ defines a morphism
$h : \SL(W) \to \SL(V)$ and a linear morphism
$b_{2d} : \Sym^d(V^{\vee}) \to \Sym^{2d}(W^{\vee})$ which is equivariant for
the previous action, \ie~$F.h(T) = b_{2d}(F).T$ for any $T \in \SL(W)$. The
identification of $V$ with $\Sym^2(W)$ also induces an embedding of $\P^1$
into $\P^2$ whose image is given by a conic $Q_0=0$ with $Q_0$ a quadratic
form with coefficients in $R$. The image of the morphism $h$ is contained in
the subgroup $\SO(Q_0)$ of elements $T \in \SL_3(R)$ such that $Q_0.T=Q_0$.

To make it concrete and keeping with the literature, when the characteristic
of $R$ is not equal to $2$, we choose the following particular basis for
$\Sym^2(W)$.  Let $u,v$ (resp. $v_1,v_2,v_3$) be a basis of $W$ (resp. $V$)
and $x,z$ the dual basis of $W^{\vee}$ (resp. $x_1,x_2,x_3$ the dual basis of
$V^{\vee}$). We let $u^2,2uv,v^2$ be a basis for $\Sym^2(W)$. Then
$b_{2d}(F)=F(x^2,2xz,z^2)$, $Q_0=x_2^2-4 x_1 x_3$ and
\begin{equation}\label{eq:GL2ToGL3}
  h
  \begin{pmatrix}
    a & b \\
    c & d
  \end{pmatrix}
  =
  \begin{pmatrix}
    a^2 &   2 a b   & b^2 \\
    a c & a d + b c & b d \\
    c^2 &   2 c d   & d^2
  \end{pmatrix}.
\end{equation}

Particularizing to the case of a quartic form $F$, a simple computation shows
that $b_8(F)$ contains valuable information on the reduction.
\begin{lemma} \label{lem:fond} Let $\Cc/\O : F=0$ with $F=Q_0^2 + \pi^{2s} G$
  be a \PHM. The model $\Cc$ is a good \PHM{} if and only if
  $\overline{b_8(G)}=\overline{G}(x^2,2xz,z^2)$ has 8 distinct roots. Moreover
  an equation of $\bar{\Cc}$ is $y^2=\overline{b_8(G)}$.
\end{lemma}

Let $F = Q_0^2 + \pi^{2s} G$ be a \PHM{} for a smooth plane quartic $C$.  The
article \cite{LRS16} gives relations between the Shioda invariants of $b_8(F)$
and the Dixmier-Ohno invariants of $F$. These expressions can be shorten
considerably in our situation by looking at the first terms of their
$\pi$-expansions.
\begin{proposition} \label{prop:shortrel} Let $F = Q_0^2+ \pi^{2s} G$ be a
  \PHM. Let $f=b_8(G)=b_8(F)/\pi^{2s}$.  The Shioda invariants $j_i$ for
  $2 \leq i \leq 7$ satisfy the congruences
  $\pi^{2s i} \cdot j_i(f) = \iota_{3 i}(F) + O(\pi^{2s(i +1)})$ where
  \begin{equation}\label{eq:DOshort}
    \parbox{0.9\textwidth}{
      \begin{dgroup*}[style={\footnotesize}]
        \begin{dmath*}
          \iota_6 = \frac{3^2}{2^5 \cdot 5 \cdot 7}\,%
          ({I_{3}}^{2}-180\,I_{6})\,,
        \end{dmath*}
        \begin{dmath*}
          \iota_9 = \frac{3^5}{2^9 \cdot 5^2 \cdot 7^{3}}\,%
          (14\,{I_{3}}^{3}-2520\,I_{3}\,I_{6}-81\,I_{9}+135\,J_{9})\,,
        \end{dmath*}
        \begin{dmath*}
          \iota_{12} = \frac{ 3^3 }{ 2^{14} \cdot 5 \cdot 7^{3}}\,%
          I_3\,(-32\,{I_{3}}^{3}+14580\,I_{3}\,I_{6}+261\,I_{9}-495\,J_{9})\,%
          + \frac{5^2}{2 \cdot 3 \cdot 7^{2}}\ \iota_6^2\,,
        \end{dmath*}
        \begin{dmath*}
          \iota_{15} = \frac{ 3^4 }{ 2^{16} \cdot 5^2 \cdot 7^{5}}\,%
          I_3\,(-592\,{I_{3}}^{4}+30780\,{I_{3}}^{2}I_{6}+2601\,I_{3}\,I_{9}-45\,I_{3}\,J_{9}+7290000\,{I_{6}}^{2}-2430\,J_{12})\,%
          + \frac{5^2}{3^2 \cdot 7}\ \iota_6\,\iota_9\,,
        \end{dmath*}
        \begin{dmath*}
          \iota_{18} = \frac{ 3^8 }{ 2^{24} \cdot 5^2 \cdot 7^4 }\,%
          I_3^2\,(-8\,{I_{3}}^{4}-14418\,{I_{3}}^{2}I_{6}-117\,I_{3}\,I_{9}+423\,I_{3}\,J_{9}+155520\,{I_{6}}^{2}-486\,I_{12})\,%
          + \frac{17^3}{2^6 \cdot 3^2 \cdot 7^3}\,{\iota_{6}}^{3}%
          + \frac{3\cdot 5}{2^5}\,\iota_{9}^{2}\,%
          - \frac{17}{2^3 \cdot 7}\,\iota_{6}\,\iota_{12}\,,%
        \end{dmath*}
        \begin{dmath*}
          \iota_{21} = \frac{ 3^7 }{ 2^{44} \cdot 5^7 \cdot 7^{5}}\,%
          I_3^{2}\,(-128\,{I_{3}}^{5}+213912\,{I_{3}}^{3}I_{6}+2961\,{I_{3}}^{2}I_{9}-8541\,{I_{3}}^{2}J_{9}-18057600\,I_{3}\,{I_{6}}^{2}+12204\,I_{3}\,I_{12}+810\,I_{3}\,J_{12}-45360\,I_{6}\,I_{9}+285120\,I_{6}\,J_{9}-4860\,I_{15}-540\,J_{15})\,%
          +{\frac {2\cdot 5^3}{3^3 \cdot 7^2}}\,\iota_{6}^{2}\iota_{9}-{\frac {13}{2\cdot 3^2}}\,\iota_{9}\,\iota_{12}-{\frac{17}{2^2\cdot 3 \cdot 7}}\,\iota_{6}\,\iota_{15}\,.
        \end{dmath*}
      \end{dgroup*}
    }
  \end{equation}

  Finally $\pi^{14\cdot 2s} \,D_{14}(f) = \iota_{42}(F) + O(\pi^{15 \cdot 2s})$ with
\begin{equation} \label{eq:trompe}
  \iota_{42} = \frac{3^{10}}{2^{18
    } \cdot 5^5}\, I_3^5 \ I_{27}\,.
\end{equation}
\end{proposition}
\medskip

The previous proposition defines a list of $6$ invariants for ternary quartics
$\ivv=(\iota_6, \ldots, \iota_{21})$ whose evaluation at a generic ternary
quartic form $F$ induces a point in $\P^{(6,9,12,15,18,21)}.$

We can now prove the direct implication of Theorem~\ref{th:main}.

\begin{theorem}[see Theorem~\ref{th:main}]
  Let $K$ be a discrete valuation field with valuation $v$, valuation ring
  $\O$ and a uniformizer $\pi$. Assume that its residue field $k$ has
  characteristic $p \ne 2,3,5,7$.  Let $C/K$ be a smooth plane quartic defined
  by $F=0$. Then $C$ has potentially good hyperelliptic reduction if and only
  if
  $$v_{\DO}(I_3(F))=0, \quad v_{\DO}(I_{27}(F))>0 \; \textrm{and} \;
  v_{\iv}(\,I_3(F)^5\, I_{27}(F)\,)=0.$$
  Under these conditions, one can also
  obtain an explicit equation for the special fiber.
\end{theorem}
\begin{proof}[Direct implication of Theorem~\ref{th:main}]
  Let $F = Q_0^2 + \pi^{2s} G$ be a good \PHM~corresponding to a stable model
  $\Cc/\O$ as in Proposition~\ref{2->1}. One has that
  $v_{\DO}(I_3(F)) = v(I_3(F))=0$ since $I_3(\bar{F})=320/9$. So
  $v_{\DO}(I_{27}(F))=v(I_{27}(F))>0$ since $\bar{F}=0$ is singular.
  Moreover, the reduction of $f=b_8(G)$ gives an equation for
  $\Cc_k : y^2=\bar{f}(x,z)$. This implies that $v(D_{14}(f))=0$ and therefore
  we get the equality
  $v(\iota_{42}(F)) = 14\cdot 2s + v(D_{14}(f)) = 14\cdot 2s.$ Since each of
  the $\iota_{3i}(F)$ is divisible by $\pi^{2si}$, we get that
  $v_{\iv}(\iota_{42}(F))=v_{\iv}(I_3(F)^5 I_{27}(F)) = 0$.\\
\end{proof}

The proof of the converse implication will keep us busy till the end of this
section. For a ternary quartic form $F \in \Sym^4(V^{\vee})$, let us denote
$\rho : \Sym^4(V^{\vee}) \to \Sym^2(V)$ the contravariant of degree $4$ and
order 2 defined in \cite{dixmier}.  If $T \in\GL_3(R)$, then $\rho$ satisfies
the relation $\rho(F.T) = \det(T)^{6} \cdot T^{-1}.\rho(F)$. Let
$\rho_0=\rho(Q_0^2)$. Then
 \begin{equation} \label{eq:rho0}
 \rho_0= - 2^{12} \cdot 3^{-2} \cdot 5 \cdot 7 \cdot (v_2^2-v_1 v_3).
 \end{equation}

 \begin{lemma} Suppose that $\cO$ is complete. Let $\W$ be a smooth scheme
   over $\cO$. Then for any $s\ge 1$, the canonical map
$$ \W(\cO)\to \W(\cO/(\pi^s))$$
is surjective.
\end{lemma}

\begin{proof}
  Indeed, for any
  $\sigma\in \W(\cO/(\pi^s))=\mathrm{Hom}_{\cO}(\Spec(\cO/(\pi^s)), \W)$ we
  can replace $\W$ by an affine open neighborhood of the image of $\sigma$ and
  it is enough to lift $\sigma$ there. So we can suppose $\W$ is affine. As
  $\cO$ is complete, the canonical map
  $$ \W(\cO)\to {\projlim}_{n\ge s} \W(\cO/(\pi^n))$$
  is an isomorphism. By the smoothness of $\W$, for any pair of integers
  $n>r>0$, the canonical map $\W(\cO/(\pi^n))\to \W(\cO/(\pi^r))$ is
  surjective.  So $\W(\cO)\to \W(\cO/(\pi^s))$ is surjective.
\end{proof}

\begin{proposition} \label{prop:rig} Suppose that $\O$ is complete.  Let
  $\rho \in \cO[v_1, v_2, v_3]$ be a quadratic (dual) form such that for some
  $s\ge 1$ we have $\rho \equiv \rho_0 \mod \pi^s$. Then there exists
  $T\in \mathrm{GL}_3(\cO)$ such that
  $$T\equiv \mathrm{Id}_3 \mod \pi^s, \quad T.\rho=u \rho_0.$$
  with $u \in \O$ such that $u \equiv 1 \mod \pi^s$.
\end{proposition}

\begin{proof}
  Let $\W=\underline{\mathrm{Isom}}_{\cO}(V_+(\rho), V_+(\rho_0))$ be the
  scheme of linear isomorphisms. For any $\cO$-algebra $A$, $\W(A)$ is the set
  of matrices $U=(x_{ij})_{1\le i, j\le 3}$ such that
  $${}^tU\rho U=y\rho_0, \quad y\in A^*$$ (here we identify the quadratic forms with their matrices
  in the canonical basis).  The scheme $\W$ is defined by $6$ equations in the
  affine space of relative dimension $10$ ($9$ for the entries of $U$ plus $1$
  for the ratio $y$). Its Jacobian matrix at the point
  $I: U=\mathrm{Id}_{3}\in M_{3\times 3}(k), y=1$ has rank $4$\footnote{For
    the explicit computations one can take for $\rho_0$ the sum of the squares
    $\sum_i v_i^2$ as the residue characteristic is different from $2$.}.
  This implies that $\W$ is smooth at this point.

  By hypothesis, there exists a $\sigma\in \W(\cO/(\pi^s))$ whose image in
  $\W(k)$ is equal to the point $I$ above. Applying the previous lemma to the
  smooth locus of $\W$ (in fact $\W$ is smooth everywhere), we find a lifting
  $T \in \W(\cO)$ of $\sigma$. Then $T \in \GL_3(\O)$ satisfies
  $$T\equiv \mathrm{Id}_3 \mod \pi^s, \quad T.\rho= u \rho_0$$
  with $u \equiv 1 \pmod{\pi^s}$.  Thus $T$ satisfies the required properties.
\end{proof}

\begin{corollary}\label{cor:T1} Suppose that $\O$ is complete and let
  $F=Q_0^2+\pi^{2s} G$ be a \PHM. There exists a matrix
  $T_1 \in\operatorname{GL}_3(\O)$ with
  $T_1 \equiv \operatorname{Id}_3\, \bmod\, \pi^{2s}$ such that
  $F_1=F.T_1=Q_0^2+\pi^{2 s} G_1$ satisfies $\rho(F_1) =\rho_0$.
 \end{corollary}

\begin{proof}
  Let $T$ be the matrix of Proposition~\ref{prop:rig} with $\rho=\rho(F) $ and
  let $F'=F.T^{-1}$. Since one has
  $T^{-1} \equiv \operatorname{Id}_3 \pmod{\pi^{2s}}$, we see that
  $F'=Q_0^2+ \pi^{2s} G'$. Moreover
  $\rho(F')=\det(T)^6 \cdot T.\rho(F) = \det(T)^6 \cdot u \cdot \rho_0.$ Let
  now $\alpha = (\det(T)^6 \cdot u)^{-1/4}$ with
  $\alpha \equiv 1 \pmod{\pi^{2s}}$. Then the form
  $F_1=\alpha F'= F.(\alpha T^{-1})$ satisfies the hypotheses.
\end{proof}

\begin{proposition}\label{newF2}
  Let $C/K: F=0$ with $F=Q_0^2 + \pi^{2s} G(x_1,x_2,x_3)$ be a \PHM{} of $C$
  such that $\rho(F) =\rho_0$.  If $b_8(G)$ is GIT-stable over $\bar{K}$, up
  to a possible extension of $K$, we can assume that $\overline{b_8(G)}$ is
  GIT-semi-stable.
\end{proposition}
\begin{proof}
  Let $f=b_8(G)$. By Corollary~\ref{cor:shabur}, after a possible extension of
  $K$, there exists a matrix $T\in\SL_2(K)$ such that $f.T=\pi^r \cdot f_1$
  with $r \in \Z$ and $f_1 \in \mathcal{O}[x,z]$ such that $\bar{f_1}=0$ is
  GIT-semi-stable.  Comparing the discriminants on both sides, we see that
  $r \geq 0$.  Since the kernel of $b_8$ is $Q_0 Q$ where $Q$ is any quadratic
  form, we can write $G.h(T)$ in a unique way as $Q_0 Q_1+G_1$ with
  $Q_1 \in K[x_1,x_2,x_3]$ a quadratic form and
  $$G_1=\beta_8 x_1^4 +\beta_7 x_1^3 x_2+ \beta_6 x_1^3 x_3 + \beta_5  x_1 x_2^3 + \beta_4 x_1 x_2^2 x_3 +\beta_3  x_2^3 x_3 +\beta_2 x_2^2 x_3^2 +\beta_1 x_2 x_3^3 + \beta_0 x_3^4 \in K[x_1,x_2,x_3].$$
  Since
  \begin{eqnarray*}
    b_8(G_1) & =& \beta_8 x^8 + 2 \beta_7 x^7 z + \beta_6 x^6 z^2 + 8 \beta_5
                  x^5 z^3 + 4 \beta_4 x^4 z^4 + 8 \beta_3 x^3 z^5 + 4 \beta_2
                  x^2 z^6 + 2 \beta_1 x z^7 + \beta_0 z^8 \\
             &=& \pi^{r} f_1,
  \end{eqnarray*}
  we get that $v(\beta_i) \geq r$ for all $i\in\{0,\ldots,8\}$, in particular
  there exists $G_2 \in \O[x_1,x_2,x_3]$ such that $G_1=\pi^r G_2$. Let us now
  write $Q_1=\pi^t Q_2$ with $t \in \Z$ and $Q_2 \in \O[x_1,x_2,x_3]$
  primitive. We have that
  $$F.h(T) = Q_0^2 + \pi^{2s} G.h(T) = Q_0^2 + \pi^{2s+t} Q_0 Q_2 + \pi^{2s+r} G_2.$$
  Since $\rho(F.h(T))=\rho(F) =\rho_0$, Lemma~\ref{a>=b} implies that
  $2s+t \geq 2s+r$ so $t \geq r$. Hence, we get that
  $$F.h(T) = Q_0^2 + \pi^{2s+r} \underbrace{(G_2+ \pi^{t-r} Q_0 Q_2)}_{ G_3}.$$
  We have that $b_8(G_3)=b_8(G_2)=b_8(G_1/\pi^r)=f_1$. This last equality also
  implies that $G_3$ is primitive.
\end{proof}

\begin{lemma}\label{a>=b}
  Let $F=Q_0^2+\pi^a Q_0Q + \pi^b G \in K[x_1,x_2,x_3]$ be a ternary quartic
  form such that $a,b \in\mathbb{Z}$, $b>0$, $Q$ is an integral primitive
  quadric and $G$ is an integral quartic form. If $\rho(F) \equiv \rho_0$
  $\pmod{\pi^{a+1}}$, then $a\geq b$.
\end{lemma}

\begin{proof}
  Assume $a<b$ and write
  $Q=a_1 x_1^2 + a_2 x_1x_2 + a_3 x_1x_3 + a_4 x_2^2 + a_5 x_2x_3 + a_6
  x_3^2$.  We see that modulo $\pi^{a+1}$ , the coefficients of $\rho(F)$ are
  equal to
\begin{equation*}
  \begin{split}
  \left( \frac{17920}{9}\pi^a a_6, -\frac{35840}{9} \pi^a a_5,
    -\frac{89600}{9}\pi^a a_4 + \frac{71680}{3}\pi^a a_3 + \frac{143360}{9},\right.\\
  \left.\frac{35840}{3}\pi^a a_4 - \frac{143360}{9}\pi^a a_3 - \frac{143360}{9},
    -\frac{35840}{9}\pi^a a_2, \frac{17920}{9}\pi^a a_1\right).
\end{split}
\end{equation*}
The equality $\rho(F) \equiv \rho_0$ $\pmod{\pi^{a+1}}$ imposes that they
should be congruent to
$$(0,0,\frac{143360}{9}, - \frac{143360}{9}, 0,0)$$ modulo $\pi^{a+1}$, from where we easily check that
$v(a_i)\geq 1$ for all $i\in\{1,\ldots,6\}$ and we get a contradiction with
$Q$ being primitive. \end{proof} We can now finish the proof of
Theorem~\ref{th:main}.
\begin{proof}[Reverse implication of Theorem~\ref{th:main}]
  Let us assume that
$$v_{\DO}(I_3(F))=0, \; v_{\DO}(I_{27}(F))>0 \text{ and } v_{\iv}(\iota_{42}(F))=0.$$
By Proposition~\ref{prop:ssred}, after a possible extension of $K$, there
exists a form $F_0$ that is $\GL_3(K)$-equivalent to $F$, integral and
primitive, and such that $v(I_3(F_0))=0$ and $v(I_{27}(F_0))>0$. Since
$v_{\iv}(\iota_{42}(F)) = v_{\iv}(\iota_{42}(F_0)) = 0$ and
$v(\iota_{42}(F_0)) = 5 v(I_3(F_0)) + v(I_{27}(F_0))>0$, the condition
$v_{\iv}(\iota_{42}(F))=0$ imposes that $v(\iota_{3i}(F_0)) > 0$ for
$i=2,\ldots,7$.  This yields these 6 equations that must be satisfied by
$\bar{F_0}$:
\begin{displaymath}
  I_{6} = \frac{1}{180}\,I_3^2,\ %
  I_{9} = \frac{49}{36}\,I_3^3,\ %
  J_{9} = \frac{49}{60}\,I_3^3,\ %
  I_{12} = \frac{343}{1620}\,I_3^4,\ %
  J_{12} = \frac{49}{36}\,I_3^4,\ %
  J_{15} + 9\,I_{15} = \frac{2744}{675}\,I_3^5.
\end{displaymath}
Substituting these expressions in the 23 relations satisfied by the
Dixmier-Ohno invariants \cite{ohno}, we obtain a polynomial system that
decomposes into 5 distinct radical components. The three first ones are such
that $I_3(\bar{F_0})=0$, and the last one is such that
$I_{27}(\bar{F_0})\neq0$, which in both cases contradicts our
hypotheses. There thus only remains one case for which $\DO(\bar{F}_0)$ are
equal to the ones in \eqref{eq:1}.  We therefore get by
Proposition~\ref{prop:MHR} that $F$ admits a \PHM{} $Q_0^2+\pi^{2r} G=0$.

Since the reduction type of the stable model does not change by completion
\cite[Prop. 10.3.15]{liu-book} and neither do the assumptions, we can take
$\O$ complete. Using Corollary~\ref{cor:T1}, we can furthermore assume that
$\rho(F) = \rho_0$. Moreover since $v_{\iota}(\iota_{42}(F))=0$,
equation~\eqref{eq:trompe} shows that $v(D_{14}(b_8(G)) \ne \infty$,
\ie~$D_{14}(b_8(G)) \ne 0$ and so $b_8(G)=0$ is GIT-stable over $\bar{K}$ (see
\cite[Prop.4.2]{mumford-fogarty}). We are therefore in the situation of
Proposition~\ref{newF2}.  In particular, we know that there exists a \PHM{}
$F_1=Q_0^2 + \pi^{2s} G_1$ such that $f_1=b_8(G_1)$ has GIT-semi-stable
reduction.  Hence, by Corollary~\ref{cor:Shiodas}, there exists
$2 \leq i_0 \leq 7$ such that $v(j_{i_0}(f_1))=0$. From
equation~\eqref{eq:DOshort}, we get that
$v(\iota_{3i_0}(F_1)) = \pi^{i_0 \cdot 2s}$.  Hence
\begin{displaymath}
  \frac{\iota_6(F_1)}{\pi^{2\cdot 2s}},\ \ldots,\ \frac{\iota_{21}(F_1)}{\pi^{7 \cdot 2s}},\ %
  \frac{\iota_{42}(F_1)}{\pi^{14 \cdot 2s}}
\end{displaymath}
is a minimal representative of the invariants $\iota$ of $F$. Since we have
assume that $v_{\iv}(\iota_{42}(F))=0$, this implies that
$v(\iota_{42}(F_1)) = 14 \cdot 2s$, so $v(D_{14}(f_1))=0$ and $F_1$ is a good
\PHM{}, which concludes the proof.
\end{proof}

Actually, one can get expressions as in Proposition~\ref{prop:shortrel} for
all the Shioda invariants $j_2,\ldots,j_{10}$. In the case of potentially good
hyperelliptic reduction, we can therefore reconstruct an equation of the
special fiber over $\bar{k}$ using the results of \cite{LR11}.

\begin{example} \label{ex:bass} We resume with Example~\ref{ex:x13}. Let
  $C/\Q : F=0$ where
$$F=(x_2 +x_3) x_1^3-(2 x_2^2 + x_3 x_2) x_1^2+(x_2^3-x_3 x_2^2+2 x_3^2 x_2-x_3^3) x_1-2 x_3^2 x_2^2+3 x_3^3 x_2.$$
One can check that $D_{27}(F)=-13^6$ so the curve has good quartic reduction
away from $13$. At $13$ the valuations of Dixmier-Ohno invariants are all $0$
except for the one of $D_{27}(F)$ which is $6$. Moreover using the formulas of
Proposition~\ref{prop:shortrel}, we get that the valuations at $13$ of the
invariants $\iota$ are $(1, 2, 2, 3, 3, 3)$. Therefore
$$v_{\DO}(I_3(F))=0, \;  v_{\DO}(D_{27}(F))=\frac{6}{27}>0, \;
v_{\iota}(I_3(F)^5 D_{27}(F))
=\frac{6}{42}-\min\left\{\frac{1}{6},\frac{2}{9},\frac{2}{12},\frac{3}{15},\frac{3}{18},\frac{3}{21},\frac{6}{42}\right\}=0.$$
The curve has hence potentially good hyperelliptic reduction at $13$ and its
special fiber is the hyperelliptic curve of affine equation $y^2=x^7-1$.
\end{example}

\begin{example} \label{ex:final} We conclude with
  Example~\ref{ex:cm}. Applying our criterion to the primes $p=37$ and
  $p=15187$, we find that $X_1$ has potentially good hyperelliptic reduction
  at these primes. Actually, at these primes it would be easy to write a good
  toggle model since $F_1$ is of the form $Q^2+ p G$. In Example~\ref{ex:X1},
  we took care of the excluded primes $5$ and $7$. For $2$, one can apply
  \cite[Prop.4.1]{KLLRSS17}: there are $2$ primes over $2$ in the maximal CM
  order of $X_1$, hence the Jacobian is absolutely simple and $X_1$ has
  potentially good hyperelliptic reduction at $2$.

  Similar analyses can be made for all other $X_i$ from \textit{loc. cit.}. We
  summarize in the Table~\ref{tab:resultsCM} our best knowledge of the
  reduction type of each of them at the primes dividing their
  discriminant. When we had to use the CM order to reach a conclusion, we
  indicate it with a $\dag$.

  \begin{table}
    \begin{footnotesize}
      \begin{center}
        \begin{tabular}{c|c|c|c|c|}
          Curve & Pot. non-hyp. &  Pot. hyp. & No pot. good & Unknown \\
          \hline\hline
          $X_{1}$ & $13$ & $2^\dagger$, $37$, $15187$ & $5$, $7$ &\\
          $X_{2}$ & & $2^\dagger$, $701$ & $3$, $7$ &\\
          $X_{3}$ & $31$ & $2^\dagger$, $233$, $356399$ & $3$, $5$ & $7$\\
          $X_{5}$ & $13$ & $2^\dagger$, $37$, $127$ & $3$ & $7$\\
          $X_{6}$ & $19$ & $2^\dagger$, $127$, $211$, $20707$ & $3$, $17$ & $7$\\
          $X_{7}$ & $73$ & $2^\dagger$, $71$, $17665559$ & $3$, $83$ & $7$\\
          $X_{8}$ & $19$ & $2^\dagger$, $499$ & $7$ &\\
          $X_{9}$ & $13$ & $79$, $233$, $857$ & $5$ & $2$, $7$\\
          $X_{10}$ & & $41$, $71$ & & $2$, $7$\\
          $X_{11}$ & $31$ & $23$, $47$, $27527$ & & $2$, $7$\\
          $X_{12}$ & & $5711$, $73064203493$ & $11$ & $7$\\
          $X_{13}$ & $43$ & $547$, $11827$, $189169$ & $2$, $11$ &\\
          $X_{14}$ & $19$ & $101$, $107$, $8378707$ & $11$ &\\
          $X_{15}$ & & $19$ & &\\
          $X_{16}$ & & $19$, $37$, $79$, $13373064392147$ & $3$ &\\
          $X_{17}$ & & $19$, $1229$, $3913841117$ & $2$ & $3$\\
          $X_{18}$ & $13$ & $19$, $101$, $251$, $7468843725186901$ & $2$ &\\
          $X_{19}$ & & $11$, $43$ & $2$ &\\
          $X_{20}$ & & $67$, $1439$, $2739021126001$ & &\\
          \hline
        \end{tabular}
      \end{center}
      \caption{CM-curves from \cite{KLLRSS17} with the reduction type for the primes dividing the discriminant} \label{tab:resultsCM}
    \end{footnotesize}
  \end{table}
\end{example}

\begin{remark} \label{rem:p>7} The assumption $p>7$ was used in various
  occasions. In Lemma~\ref{lem:test-conic} or in
  Proposition~\ref{prop:shortrel}, a better choice of invariants or
  normalization of the expressions may overcome this restriction, but the
  appearance of these primes in equation \eqref{eq:rho0} is more
  symptomatic. Let us for instance consider $p=7$. It is known that the Klein
  quartic $C/\Q : F=0$ where $F=x_1^3 x_2+x_2^3x_3+x_3^3 x_1$ has potentially
  good hyperelliptic reduction modulo $7$. However, since
  $\Aut_{\bar{\Q}}(C) \simeq \PSL_2(\F_7)$ and the latter is not a subgroup of
  $\SO_3(\bar{\Q}) \simeq \PSL_2(\bar{\Q})$, it cannot be the automorphism
  group of any non-degenerate conic. Hence, $\rho_0(F)=0$ and this is the case
  for any other possible covariant or contravariant of order 2. Therefore, we
  cannot go further with our current strategy in this case. However, in this
  particular case it is easy to find a \PHM, see \cite[p.56 and p.82]{Elkies}.
\end{remark}

\newcommand{\etalchar}[1]{$^{#1}$}


\begin{thebibliography}{KlcLGS20}

\bibitem[ACGH85]{harris}
E.~Arbarello, M.~Cornalba, P.~Griffiths, and J.~Harris.
\newblock {\em Geometry of algebraic curves, Vol. {I}}, volume 267.
\newblock Grundlehren der Mathematischen Wissenschaften, Springer-Verlag,
  New-York, 1985.

\bibitem[Art09]{artebani}
M.~Artebani.
\newblock A compactification of {$M_3$} via {$K3$} surfaces.
\newblock {\em Nagoya Math. J.}, 196:1--26, 2009.

\bibitem[Bas15]{basson}
R.~Basson.
\newblock {\em Arithm{\'e}tique des espaces de modules des courbes
  hyperelliptiques de genre $3$ en caract\'eristique positive}.
\newblock PhD thesis, Universit\'e de Rennes 1, Rennes, 2015.

\bibitem[BBW17]{BBW17}
M.~B\"{o}rner, I.~I. Bouw, and S.~Wewers.
\newblock Picard curves with small conductor.
\newblock In {\em Algorithmic and experimental methods in algebra, geometry,
  and number theory}, pages 97--122. Springer, Cham, 2017.

\bibitem[BCL{\etalchar{+}}15]{WINE1}
I.~Bouw, J.~Cooley, K.~Lauter, E.~{Lorenzo Garc\'\i a}, M.~Manes, R.~Newton,
  and E.~Ozman.
\newblock Bad reduction of genus three curves with complex multiplication.
\newblock In {\em Women in numbers {E}urope}, volume~2 of {\em Assoc. Women
  Math. Ser.}, pages 109--151. Springer, Cham, 2015.

\bibitem[BDM{\etalchar{+}}19]{BDMTV17}
J.~Balakrishnan, N.~Dogra, J.~S. M\"{u}ller, J.~Tuitman, and J.~Vonk.
\newblock Explicit {C}habauty-{K}im for the split {C}artan modular curve of
  level 13.
\newblock {\em Ann. of Math. (2)}, 189(3):885--944, 2019.

\bibitem[BKSW20]{BKKSW}
I.~I. Bouw, A.~Koutsianas, J.~Sijsling, and S.~Wewers.
\newblock Conductor and discriminant of picard curves, 2020.
\newblock To appear in J. of the London Math. Soc.

\bibitem[Bur92]{Bur}
J.-F. Burnol.
\newblock Remarques sur la stabilit\'{e} en arithm\'{e}tique.
\newblock {\em Internat. Math. Res. Notices}, (6):117--127, 1992.

\bibitem[BW17]{BW}
I.~I. Bouw and S.~Wewers.
\newblock Computing {$L$}-functions and semistable reduction of superelliptic
  curves.
\newblock {\em Glasg. Math. J.}, 59(1):77--108, 2017.

\bibitem[Cle80]{clemens}
C.~H. Clemens.
\newblock {\em A scrapbook of complex curve theory}.
\newblock Plenum Press, New York-London, 1980.
\newblock The University Series in Mathematics.

\bibitem[DDMM19]{maistret}
T.~Dokchitser, V.~Dokchitser, C.~Maistret, and A.~Morgan.
\newblock Semistable types of hyperelliptic curves.
\newblock In {\em Algebraic curves and their applications}, volume 724 of {\em
  Contemp. Math.}, pages 73--135. Amer. Math. Soc., Providence, RI, 2019.

\bibitem[Dem12]{demazure}
M.~Demazure.
\newblock R\'esultant, discriminant.
\newblock {\em Enseign. Math. (2)}, 58(3-4):333--373, 2012.

\bibitem[Dix87]{dixmier}
J.~Dixmier.
\newblock On the projective invariants of quartic plane curves.
\newblock {\em Adv. in Math.}, 64:279--304, 1987.

\bibitem[DK02]{DerKem}
H.~Derksen and G.~Kemper.
\newblock {\em Computational invariant theory}.
\newblock Invariant Theory and Algebraic Transformation Groups, I.
  Springer-Verlag, Berlin, 2002.
\newblock Encyclopaedia of Mathematical Sciences, 130.

\bibitem[Edi90]{edixhoven90}
B.~Edixhoven.
\newblock {Minimal resolution and stable reduction of {$X_0(N)$}.}
\newblock {\em {Ann. Inst. Fourier}}, 40(1):31--67, 1990.

\bibitem[Elk99]{Elkies}
N.~D. Elkies.
\newblock The {K}lein quartic in number theory.
\newblock {\em Math. Sci. Res. Inst. Publ.}, 35:51--101, 1999.

\bibitem[Els15]{elsenhans}
A.-S. Elsenhans.
\newblock Explicit computations of invariants of plane quartic curves.
\newblock {\em J. Symbolic Comput.}, 68(part 2):109--115, 2015.

\bibitem[ES20]{elsenhans-stoll}
A.-S. Elsenhans and M.~Stoll.
\newblock Good models for cubic surfaces.
\newblock {\em In progress}, 2020.

\bibitem[GKZ94]{gelfand}
I.~M. Gelfand, M.~M. Kapranov, and A.~V. Zelevinsky.
\newblock {\em Discriminants, resultants, and multidimensional determinants}.
\newblock Mathematics: Theory \& Applications. Birkh\"auser Boston Inc.,
  Boston, MA, 1994.

\bibitem[Gla79]{glass}
J.~P. Glass.
\newblock Bitangents of plane quartics.
\newblock {\em Bull. Austral. Math. Soc.}, 20(2):207--210, 1979.

\bibitem[GLL15]{GLL}
O.~Gabber, Q.~Liu, and D.~Lorenzini.
\newblock Hypersurfaces in projective schemes and a moving lemma.
\newblock {\em Duke Math. J.}, 164(7):1187--1270, 2015.

\bibitem[Har87]{harriscm}
J.~Harris.
\newblock Curves and their moduli.
\newblock In {\em Algebraic geometry, {B}owdoin, 1985 ({B}runswick, {M}aine,
  1985)}, volume~46 of {\em Proc. Sympos. Pure Math.}, pages 99--143. Amer.
  Math. Soc., Providence, RI, 1987.

\bibitem[Has99]{hassett}
B.~Hassett.
\newblock Stable log surfaces and limits of quartic plane curves.
\newblock {\em Manuscripta Math.}, 100(4):469--487, 1999.

\bibitem[HL10]{hyeon}
D.~Hyeon and Y.~Lee.
\newblock Log minimal model program for the moduli space of stable curves of
  genus three.
\newblock {\em Math. Res. Lett.}, 17(4):625--636, 2010.

\bibitem[HLP00]{howe00}
E.~W. {Howe}, F.~{Lepr\'evost}, and B.~{Poonen}.
\newblock {Large torsion subgroups of split Jacobians of curves of genus two or
  three.}
\newblock {\em {Forum Math.}}, 12(3):315--364, 2000.

\bibitem[HM98]{harris-moduli}
J.~Harris and I.~Morrison.
\newblock {\em Moduli of curves}, volume 187 of {\em Graduate Texts in
  Mathematics}.
\newblock Springer-Verlag, New York, 1998.

\bibitem[Hol95]{Hol}
R.-P. Holzapfel.
\newblock {\em The ball and some {H}ilbert problems}.
\newblock Lectures in Mathematics ETH Z\"urich. Birkh\"auser Verlag, Basel,
  1995.
\newblock Appendix I by J. Estrada Sarlabous.

\bibitem[Kan00]{kang}
P.-L. Kang.
\newblock On singular plane quartics as limits of smooth curves of genus three.
\newblock {\em J. Korean Math. Soc.}, 37(3):411--436, 2000.

\bibitem[KlcLGS20]{KLS18}
P.~n. K\i~l\i \c{c}er, E.~Lorenzo~Garc\'{\i}a, and M.~Streng.
\newblock Primes dividing invariants of {CM} {P}icard curves.
\newblock {\em Canad. J. Math.}, 72(2):480--504, 2020.

\bibitem[KLL{\etalchar{+}}16]{WINE1-2}
P.~K{\i}l{\i}\c{c}er, K.~Lauter, E.~{Lorenzo Garc{\'i}a}, R.~Newton, E.~Ozman,
  and M.~Streng.
\newblock A bound on the primes of bad reduction of {CM} curves of genus 3.
\newblock 2016.

\bibitem[KLL{\etalchar{+}}18]{KLLRSS17}
P.~K{\i}l\i\c{c}er, H.~Labrande, R.~Lercier, C.~Ritzenthaler, J.~Sijsling, and
  M.~Streng.
\newblock Plane quartics over {$\Bbb {Q}$} with complex multiplication.
\newblock {\em Acta Arith.}, 185(2):127--156, 2018.

\bibitem[Kol97]{kollar}
J.~Koll\'{a}r.
\newblock Polynomials with integral coefficients, equivalent to a given
  polynomial.
\newblock {\em Electron. Res. Announc. Amer. Math. Soc.}, 3:17--27, 1997.

\bibitem[Kon00]{kondo}
S.~Kond\=o.
\newblock A complex hyperbolic structure for the moduli space of curves of
  genus three.
\newblock {\em J. Reine Angew. Math.}, 525:219--232, 2000.

\bibitem[KW05]{kowe}
K.~Koike and A.~Weng.
\newblock Construction of {CM} {P}icard curves.
\newblock {\em Math. Comp.}, 74(249):499--518 (electronic), 2005.

\bibitem[Liu93]{liug2}
Q.~Liu.
\newblock Courbes stables de genre {$2$} et leur sch\'ema de modules.
\newblock {\em Math. Ann.}, 295(2):201--222, 1993.

\bibitem[Liu02]{liu-book}
Q.~Liu.
\newblock {\em Algebraic geometry and arithmetic curves}, volume~6 of {\em
  Oxford Graduate Texts in Mathematics}.
\newblock Oxford University Press, Oxford, 2002.
\newblock Translated from the French by Reinie Ern\'e, Oxford Science
  Publications.

\bibitem[LL99]{LiLo}
Q.~Liu and D.~Lorenzini.
\newblock Models of curves and finite covers.
\newblock {\em Compositio Math.}, 118:61--102, 1999.

\bibitem[LLLR19]{HSOPList}
R.~{Lercier}, Q.~{Liu}, E.~{Lorenzo Garc{\'i}a}, and C.~{Ritzenthaler}.
\newblock Reduction type of smooth quartics, ancillary files.
\newblock Available at \href{https://arxiv.org/src/1803.05816/anc}{\tt
  https://arxiv.org/src/1803.05816/anc/}, 2019.

\bibitem[LR12]{LR11}
R.~Lercier and C.~Ritzenthaler.
\newblock Hyperelliptic curves and their invariants: geometric, arithmetic and
  algorithmic aspects.
\newblock {\em J. Algebra}, 372:595--636, 2012.

\bibitem[LR19]{LR19}
R.~Lercier and C.~Ritzenthaler.
\newblock Siegel modular forms of degree three and invariants of ternary
  quartics.
\newblock {\em Proc. Amer. Math. Soc.}, 2019.
\newblock To appear.

\bibitem[LRS18]{LRS16}
R.~{Lercier}, C.~{Ritzenthaler}, and J.~{Sijsling}.
\newblock {Reconstructing plane quartics from their invariants}.
\newblock {\em Discrete \& Computational Geometry}, pages 1--41, 2018.

\bibitem[LS16]{LaSo}
J.-C. Lario and A.~{Somoza}.
\newblock A note on {P}icard curves of {CM}-type.
\newblock Preprint, 2016.

\bibitem[MF82]{mumford-fogarty}
D.~Mumford and J.~Fogarty.
\newblock {\em Geometric invariant theory}, volume~34 of {\em Ergebnisse der
  Mathematik und ihrer Grenzgebiete [Results in Mathematics and Related
  Areas]}.
\newblock Springer-Verlag, Berlin, second edition, 1982.

\bibitem[Mum77]{mum-ens}
D.~Mumford.
\newblock Stability of projective varieties.
\newblock {\em Enseignement Math. (2)}, 23(1-2):39--110, 1977.

\bibitem[Ohn07]{ohno}
T.~Ohno.
\newblock The graded ring of invariants of ternary quartics {I}, 2007.
\newblock unpublished.

\bibitem[Ses77]{Se}
C.~S. Seshadri.
\newblock Geometric reductivity over arbitrary base.
\newblock {\em Advances in Math.}, 26(3):225--274, 1977.

\bibitem[Sha80]{shah}
J.~Shah.
\newblock A complete moduli space for {$K3$}\ surfaces of degree {$2$}.
\newblock {\em Ann. of Math. (2)}, 112(3):485--510, 1980.

\bibitem[Shi67]{shioda67}
T.~Shioda.
\newblock On the graded ring of invariants of binary octavics.
\newblock {\em American J. of Math.}, 89(4):1022--1046, 1967.

\bibitem[Shi93]{shioda-quartic}
T.~Shioda.
\newblock Plane quartics and {M}ordell-{W}eil lattices of type {$E_7$}.
\newblock {\em Comment. Math. Univ. St. Paul.}, 42(1):61--79, 1993.

\bibitem[Sil98]{silverman-dyna}
J.~H. Silverman.
\newblock The space of rational maps on {$\bold P^1$}.
\newblock {\em Duke Math. J.}, 94(1):41--77, 1998.

\bibitem[STW14]{szpiro-semi}
L.~Szpiro, M.~Tepper, and P.~Williams.
\newblock Semi-stable reduction implies minimality of the resultant.
\newblock {\em J. Algebra}, 397:489--498, 2014.

\bibitem[Tsu86]{tsuyumine-g3}
S.~Tsuyumine.
\newblock On {S}iegel modular forms of degree three.
\newblock {\em Amer. J. Math.}, 108(4):755--862, 1986.

\bibitem[Yuk86]{yukie}
A.~Yukie.
\newblock {\em A{pplications} {of} {Equivariant} {Morse} {Stratifications}
  ({Geometric} {Invariant} {Theory}, {Eisenstein} {Series})}.
\newblock ProQuest LLC, Ann Arbor, MI, 1986.
\newblock Thesis (Ph.D.)--Harvard University.

\end{thebibliography}
\end{document}